\documentclass{article}

\usepackage[english]{babel}
\usepackage[utf8]{inputenc}
\usepackage{amsmath}
\usepackage{amssymb}
\usepackage{nicematrix}
\usepackage{graphicx}
\usepackage{graphicx}
\usepackage[colorlinks=true, allcolors=blue]{hyperref}
\usepackage{setspace}
\usepackage{amssymb}
\usepackage{mathtools}
\usepackage{multirow, array}
\usepackage{pgfplots}
\usepackage[numbers,sort]{natbib}
\usepackage{pdflscape}

\pgfplotsset{compat=1.18}
\usepgfplotslibrary{fillbetween}
\usepackage{endnotes}
\usepackage{amsthm}
\usepackage{enumitem}
\usepackage[]{mdframed}
\usepackage{blindtext}
\usepackage{titlesec}
\usepackage{bm,bbm}
\usepackage{comment}
\usepackage{caption}
\usepackage[font=scriptsize]{subcaption}
\usepackage[margin=1in]{geometry}
\usepackage{cleveref}
\usepackage{color}
\usepackage[normalem]{ulem}

\usepackage{nicematrix}

\newtheorem{theorem}{Theorem}
\newtheorem{proposition}{Proposition}

\newtheorem{corollary}{Corollary}

\newtheorem{remark}{Remark}

\theoremstyle{remark}

\newcommand{\conv}{\textup{conv}}
\newcommand{\clconv}{\textup{cl conv}}
\newcommand{\supp}{\textup{supp}}
\newcommand{\Diag}{\textup{Diag}}
\newcommand{\cl}{\textup{cl}}

\newcommand{\R}{\ensuremath{\mathbb{R}}}

\newcommand{\Cp}[1]{\ensuremath{\mathcal{C}_+^{#1}}}

\newcommand\Cbracket[1]{\left\{#1\right\}}
\newcommand\Cparen[1]{\left(#1\right)}

\newcommand{\innerProd}[2]{\left\langle#1,#2\right\rangle}
\DeclareMathOperator*{\proj}{proj}

\usepackage[colorinlistoftodos,prependcaption,textsize=tiny]{todonotes}

\title{Rank-one convexification for quadratic optimization problems with step function penalties}
\author{Soobin Choi\thanks{Department of Industrial and Systems Engineering, University of Southern California, \texttt{soobinch@usc.edu}}, Valentina Cepeda\thanks{Departamento de Ingenier\'ia Industrial, Universidad de los Andes, \texttt{v.cepeda@uniandes.edu.co}}, Andr\'es G\'omez\thanks{Department of Industrial and Systems Engineering, University of Southern California, \texttt{gomezand@usc.edu}} and Shaoning Han\thanks{Department of Mathematics, and Institute of Operations Research and Analytics, National University of Singapore, \texttt{shaoninghan@nus.edu.sg}}}

\date{\today}

\begin{document}
\maketitle

\begin{abstract}
We investigate convexification for convex quadratic optimization with step function penalties. Such problems can be cast as mixed-integer quadratic optimization problems, where binary variables are used to encode the non-convex step function. First, we derive the convex hull for the epigraph of a quadratic function defined by a rank-one matrix and step function penalties. Using this rank-one convexification, we develop copositive and semi-definite relaxations for general convex quadratic functions. Leveraging these findings, we construct convex formulations to the support vector machine problem with 0--1 loss and show that they yield robust estimators in settings with anomalies and outliers.
\end{abstract}

\paragraph{Statements and declarations} The authors have no competing interests to declare that are relevant to the content of this article.

\section{Introduction}
Given vectors $\bm{a},\bm{c}\in \R^n$ and positive semi-definite matrix $\bm{Q}\in \R^{n\times n}$, we consider mixed-integer quadratic optimization problems of the form 
\begin{subequations}\label{prob:two-sided}
    \begin{align}
        \min_{\bm{x}\in \R^n,\bm{z}\in \{0,1\}^n}\;      & \bm{x}^\top \bm{Q}\bm{x} + \bm{a}^\top \bm{x} + \bm{c}^\top \bm{z} \label{eq:two-sided_obj}\\
        \text{s.t.}\; & x_iz_i \geq 0,\; x_i(1-z_i)\leq 0    & i=1,\ldots,n,\label{eq:two-sided_sign}\\
            &(\bm{x},\bm{z})\in \mathcal{F}\subseteq \mathbb{R}^n\times \{0,1\}^n.\label{eq:two-sided_vars}
    \end{align}
\end{subequations}
Constraints \eqref{eq:two-sided_sign}, along with binary constraints $\bm{z}\in \{0,1\}^n$, enforce the logical conditions $z_i=1\implies x_i\geq 0$ and $z_i=0\implies x_i\leq 0$ \textemdash we refer to such constraints as \emph{sign indicator constraints}. Observe that constraints \eqref{eq:two-sided_sign} can be linearized as
\begin{equation}\label{eq:mixing}
x_i+M(1-z_i)\geq 0,\; x_i-Mz_i\leq 0,\quad i=1,\dots,n,
\end{equation}
where $M$ is a sufficiently large number. 
If $\mathcal{F}=\mathbb{R}^n\times \{0,1\}^n$, then the optimization problem \eqref{prob:two-sided} is equivalent to imposing non-convex, discontinuous penalties on the continuous variables $\bm{x}$ of the form $h(x_i)=\begin{cases}c_i&\text{if }x_i>0\\0&\text{if }x_i\leq 0\end{cases}$ if $\bm{c}\geq \bm{0}$. One of our main motivations for studying \eqref{prob:two-sided} is to tackle \emph{support vector machines} (SVMs) with the 0-1 loss, a class of problems we discuss at length in \S\ref{sec:applications}. 

In this paper, we study the interplay between these constraints and the nonlinear convex quadratic term in the objective \eqref{eq:two-sided_obj} involving the continuous variables to derive strong convex relaxations. 
Specifically, we focus our attention on two related mixed-integer sets associated with the problem \eqref{prob:two-sided}.
The first one is the epigraph of the mixed-integer quadratic form, defined as 
\begin{equation*}
    \mathcal{X}_{\bm{Q}}:=\{(\bm{x},\bm{z},t)\in\mathbb{R}^n\times\{0,1\}^n\times\mathbb{R}:t\geq \bm{x}^\top \bm{Q}\bm{x},~x_iz_i\geq 0,~x_i(1-z_i)\leq 0, ~i=1,\ldots,n\},
\end{equation*}
and its ``one-sided" relaxation 
\begin{align}\label{eq:oneSidedFirst}
    \mathcal{X}_{\bm{Q}}^-:=\{(\bm{x},\bm{z},t)\in\mathbb{R}^n\times\{0,1\}^n\times\mathbb{R}:t\geq \bm{x}^\top \bm{Q}\bm{x},~x_i(1-z_i)\leq 0, ~i=1,\ldots,n\}.
\end{align}
Indeed, a stronger relaxation of \eqref{prob:two-sided} than the one obtained from the big-M constraints \eqref{eq:mixing} is
\begin{align}\label{eq:formulationHull1}
        \min_{(\bm{x},\bm{z},t)\in \R^{2n+1}}\; t + \bm{a}^\top \bm{x} + \bm{c}^\top \bm{z}\;\;
        \text{s.t.}\;  (\bm{x},\bm{z},t)\in \text{cl conv}(\mathcal{X}_{\bm{Q}}),\;(\bm{x},\bm{z})\in \clconv(\mathcal{F}),
    \end{align}
    and this relaxation is exact for all $\bm{a},\bm{c}\in \R^n$ if $\mathcal{F}=\R^n\times \{0,1\}^n$. Naturally, since problem \eqref{prob:two-sided} includes several NP-hard problems as special cases even if $\mathcal{F}=\R^n\times \{0,1\}^n$, it may not be possible to obtain tractable descriptions of $\text{cl conv}(\mathcal{X}_{\bm{Q}})$ unless $\bm{Q}$ has a special structure or $n$ is fixed. 

The second, more general set we study is
\begin{align}\label{eq:defBarX}\mathcal{\bar X}:=\big\{(\bm{x},\bm{X},\bm{z})\in\mathbb{R}^n\times\R^{n\times n}\times\{0,1\}^n:&\;\bm{X}\succeq \bm{xx^\top},
 \;x_iz_i\geq 0,~x_i(1-z_i)\leq 0,\;i=1,\dots,n\big\}.\end{align}
 Indeed a relaxation of \eqref{prob:two-sided} using $\mathcal{\bar X}$ is 
 \begin{align*}
        \min_{(\bm{x},\bm{X},\bm{z})\in \R^{n^2+2n}}\; \langle \bm{Q},\bm{X}\rangle + \bm{a}^\top \bm{x} + \bm{c}^\top \bm{z}\;\;
        \text{s.t.}\;  (\bm{x},\bm{X},\bm{z})\in \text{cl conv}(\mathcal{\bar X}),\;(\bm{x},\bm{z})\in \clconv(\mathcal{F}),
    \end{align*}
    where $\langle \bm{Q},\bm{X}\rangle =\sum_{i=1}^n\sum_{j=1}^n Q_{ij}X_{ij}$ is the usual inner product between matrices, and this relaxation is exact for all $\bm{a},\bm{c}\in \R^n$ and positive semi-definite matrices $\bm{Q}\succeq \bm{0}$ if $\mathcal{F}=\R^n\times \{0,1\}^n$. Note that since the relaxation $\text{cl conv}(\mathcal{\bar X})$ guarantees optimality for \emph{all} positive semidefinite matrices, whereas the relaxation in \eqref{eq:formulationHull1} is valid for a specific matrix only, the set $\text{cl conv}(\mathcal{\bar X})$ is substantially more general and harder to describe. Set $\mathcal{\bar X}$ can also be used to jointly convexify several quadratic constraints, since 
    \begin{align*}
    &\bm{x^\top Q_j x}+\bm{a_j^\top x}+\bm{c_j^\top z}\leq b_j,\; \forall j\in \{1,\dots,m\}\text{ and }\eqref{eq:two-sided_sign}\\
    \implies&\exists \bm{X}\text{ such that } (\bm{x},\bm{X},\bm{z})\in \clconv(\bar{\mathcal{X}})\text{ and } \langle\bm{Q_j},\bm{X}\rangle+\bm{a_j^\top x}+\bm{c_j^\top z}\leq b_j,\; \forall j\in \{1,\dots,m\}.
    \end{align*}

\paragraph{Contributions and outline}
In this paper, we provide an explicit description of the closure of the convex hull of $\mathcal{X}_{\bm{Q}}$ when $\bm{Q}$ is rank-one. Leveraging our insights from the rank-one case, we also develop strong convex (copositive- and SDP-representable) relaxations for set $\mathcal{\bar X}$. We show how to effectively implement the proposed relaxations in the special case of SVM problems, and discuss how the proposed relaxations compare and improve upon existing approaches in the statistical/machine learning literature to solve non-convex SVMs.

The rest of the paper is organized as follows. In \S\ref{sec:applications} we discuss support vector machine learning problems. In \S\ref{sec:litReview} we review the relevant literature for the paper. In \S\ref{sec:hull} we derive an ideal description of $\cl\;\conv(\mathcal{X}_{\bm{Q}})$ when $\bm{Q}$ is rank-one, and in \S\ref{sec:validExtended} we propose valid inequalities for $\cl\;\conv(\bar{ \mathcal{X}})$. In \S\ref{sec:SVM} we discuss how to efficiently implement the derived inequalities for support vector machine problems. Finally, we present the results of our computational experiments in \S\ref{sec:computations} and conclude the paper in \S\ref{sec:conclusions}.

\paragraph{Notation} Given $n\in \mathbb{Z}_+$, let $[n]=\{1,\ldots,n\}$, and let $\bm{0}_n$ and $\bm{1}_n$ denote an $n$-dimensional vector of zeros and ones, respectively; we omit the subscript $n$ if the dimension is clear from the context. Similarly, we let $\bm{0}_{n\times n}$ and $\bm{I}_{n\times n}$ denote $(n\times n)$-dimensional matrices of zeros and the identity respectively, and omit the subscripts when clear from the context. For a vector $\bm{v}\in\mathbb{R}^n$, define $\supp(\bm{v})=\left\{i\in [n]: v_i\neq 0\right\}$ as the support of $\bm{v}$, and define $\supp_+(\bm{v})=\left\{i\in [n]: v_i> 0\right\}$ and $\supp_-(\bm{v})=\left\{i\in [n]: v_i< 0\right\}$ as the set of indexes with positive and negative signs, respectively. Moreover, given $S\subseteq [n]$ and $v\in \R^n$, we denote by $\bm{v_S}\in \R^S$ the subvector induced by the support $S$. Given a square matrix $\bm V\in \R^{n\times n}$ and $S\subseteq [n]$, we let $\bm{V_{S}}\in \R^{S\times S}$ denote the submatrix of $\bm{V}$ induced by the rows and columns of $S$. Moreover, for a (not-necessarily square) matrix $\bm{V}\in \R^{n\times m}$ and $S\subseteq [n]$, we let $\bm{V_{S,m}}\in \R^{S\times m}$ be the submatrix of $\bm{V}$ induced by the rows of $S$.

We denote by $\mathcal{S}_+^n$ and $\Cp{n}$ the cones of $n$-dimensional symmetric positive semi-definite and copositive matrices, that is, 
\begin{align*}
  \mathcal{S}_+^n&=\{\bm{X}\in\R^{n\times n}: \bm{X}=\bm{X^\top},\; \bm{v}^\top \bm{X} \bm{v}\geq 0 \text{ for all } \bm{v}\in\mathbb{R}^n\},\text{ and }\\
  \Cp{n}&=\{\bm{X}\in\R^{n\times n}:\bm{X}=\bm{X^\top},\; \bm{v}^\top \bm{X} \bm{v}\geq 0 \text{ for all } \bm{v}\in \R_+^n\}.
\end{align*} 
For simplicity, given a matrix $\bm{A}\in \R^{n\times n}$, we sometimes use the notation $\bm{A}\succeq \bm{0}$ to denote $\bm{A}\in \mathcal{S}_+^n$.

The inner product of two matrices is denoted by $\langle\bm{A},\bm{B}\rangle=\sum_{i,j=1}^n A_{ij}B_{ij}$. Given a vector $\bm{v}\in\mathbb{R}^{n}$, $\Diag(\bm{v})\in\mathbb{R}^{n\times n}$ denotes the diagonal matrix with its $i$th diagonal entry $v_i$ for $i=1,\ldots,n$. 
We define the two operators $(c)_+=\max\{c,0\}$ and $(c)_-=\min\{c,0\}$ for $c\in\mathbb{R}$. For a set $\mathcal{X}\subseteq \mathbb{R}^n$, we denote its convex hull by $\conv(\mathcal{X})$ and the closure of the convex hull by $\clconv(\mathcal{X})$. We define $0/0=0$ and $0\cdot\infty=0$ by convention. We let $\mathbbm{1}_{\{\cdot\}}$ denote the indicator function that is equal to one whenever the condition holds and is zero otherwise; for example, $\mathbbm{1}_{\{x> 0\}}=1$ if $x> 0$ and is $0$ if $x\leq 0$. 

\section{Robust support vector machines}\label{sec:applications}

In this section, we discuss a specific application of \eqref{prob:two-sided}, which is the main focus of the computational experiments in this paper. 
Given labeled data $\{(\bm{a_i},y_i)\}_{i=1}^n$, where $\bm{a_i}\in \R^p$ encodes the features of point $i$ and $y_i\in \{-1,1\}$ denotes its class, consider the support vector machine (SVM) problem defined in the seminal paper by 
\citet{cortes1995support}, given by
\begin{subequations}\label{eq:SVM}
	\begin{align}
		\min_{\bm{w}\in \R^p,\bm{\xi}\in \R_+^n}\;&\|\bm{w}\|_2^2+\lambda\sum_{i=1}^n \xi_i^\eta\\
		\text{s.t.}\:&y_i\left(\bm{a_i^\top w}\right)\geq 1-\xi_i\quad \forall i\in \{1,\dots,n\}
	\end{align}
	
\end{subequations}
\noindent for some regularization parameter $\lambda>0$ and some sufficiently small constant $\eta\in \R_+$. \citet{cortes1995support} argue that if $\lambda$ is sufficiently large and $\eta$ is sufficiently small, then an optimal solution of \eqref{eq:SVM} yields a hyperplane that misclassifies the least amount of points possible and has maximum margin among all hyperplanes with minimal misclassification. The extreme case where $\eta=0$ leads to the SVM problem with 0--1 misclassification loss given by 
\begin{align}\label{eq:SVM01}
\min_{\bm{w}\in \R^p}\;\|\bm{w}\|_2^2+\lambda\sum_{i=1}^n \mathbbm{1}_{\{y_i\bm{a_i^\top w}< 1\}}.
\end{align}
However, noting that \eqref{eq:SVM} is NP-hard if $\eta<1$, the authors instead suggest using $\eta=1$, resulting in the popular SVM with the hinge loss
\begin{align}\label{eq:hinge}
	\min_{\bm{w}\in \R^p}\;&\|\bm{w}\|_2^2+\lambda\sum_{i=1}^n \max\left\{0,1-y_i\left(\bm{a_i^\top w}\right)\right\}.
\end{align}
Setting $\eta>1$ leads to other convex surrogates of SVM problems that have also been studied in the literature \cite{chang2008coordinate,lee2001ssvm}.

Typical convex surrogates of \eqref{eq:SVM01} such as the hinge loss problem \eqref{eq:hinge} have been shown to perform well in low noise settings \citep{bartlett2006convexity}. However, these surrogates are known to perform worse than the 0-1 loss in the presence of uncertainty or outliers \citep{manwani2013noise,ghosh2015making}.
In particular, label noise, characterized by label flips in data, affects the performance of classification problems. The standard hinge loss SVM algorithm \eqref{eq:hinge} tends to penalize outliers more and, as a result, these points strongly influence the placement of the separating hyperplane, leading to low classification performance. To improve SVM robustness to label uncertainty, various approaches have been proposed in the literature. \citet{song2002robust} suggested evaluating an adaptive margin based on the distance of each point to its class center to mitigate the influence of points far from class centroids. 
Other studies have focused on replacing the hinge loss function with non-convex alternatives. Some popular choices are the ramp loss \citep{wu2007robust}, the rescaled hinge loss \citep{xu2017robust}, the $\psi$-learning loss \citep{shen2003psi} or simply using the exact 0-1 loss. 
Unfortunately, while improving robustness, these methods may result in prohibitive increases in computational time due to the non-convex nature of the optimization problems. For example, \citet{brooks2011support} proposes to use mixed-integer optimization (MIO) solvers to tackle SVMs with the ramp or 0-1 loss using big-M formulations: the resulting approach, which we review next, does not scale beyond $n\approx 100$. Thus, heuristics or local minimization methods that do not produce global minimizers are typically used in practice~\cite{wang2021support}. 

\subsection{Natural big-M formulation }
 \label{sec:bigM}
To solve SVMs with the exact misclassification loss, \citet{brooks2011support} proposes to introduce binary variables $z_i=1$ if point $i$ lies in the margin or is misclassified, and $z_i=0$ if the point is correctly 
classified. Then, they reformulate \eqref{eq:SVM01} as
\begin{subequations}\label{eq:BigMPenalty}
	\begin{align}
		\min_{\bm{w},\,\bm z}\;&\|\bm{w}\|_2^2+\lambda\sum_{i=1}^n z_i\label{eq:BigMPenalty_obj}\\
		\text{s.t.}\:&y_i\left(\bm{a_i^\top w}\right)\geq 1-Mz_i\quad \forall i\in [n],\label{eq:BigMPenalty_classification}\\
		&\bm{w}\in \R^p,\;\bm{z}\in \{0,1\}^n\label{eq:BigMPenalty_bounds}
	\end{align}
\end{subequations}
where $M$ is a sufficiently large number (see also \cite{jammal2020robust} for a similar formulation). If $z_i=0$, constraint \eqref{eq:BigMPenalty_classification} forces point $i$ to be correctly classified, while setting $z_i=1$ incurs a fixed cost of $\lambda$ but allows point $i$ to be in the margin or in the ``incorrect'' side of the hyperplane. \citet{brooks2011support} also shows that the estimator obtained from solving \eqref{eq:BigMPenalty_obj}-\eqref{eq:BigMPenalty_bounds} to optimality is consistent under appropriate conditions. 

\paragraph{Relaxation quality} A good proxy for the effectiveness of branch-and-bound methods is the relaxation quality of the MIO formulation, obtained by relaxing the binary constraints to bound constraints $\bm{0}\leq \bm{z}\leq \bm{1}$.
Consider the relaxation of \eqref{eq:BigMPenalty}:

	\begin{align}\label{eq:BigMPenaltyRelaxation}
		\min_{\bm{w}\in \R^p,\bm{0}\leq\bm{z}\leq \bm{1}}\;&\|\bm{w}\|_2^2+\lambda\sum_{i=1}^n z_i
\text{ 
 s.t.  }\:y_i\left(\bm{a_i^\top w}\right)\geq 1-Mz_i\quad \forall i\in [n].
	\end{align}
\normalsize
A common approach to assess the relaxation quality is through the $\texttt{gap}:=\left(\zeta_{mio}-\zeta_{relax}\right)/\zeta_{mio}$, where $\zeta_{mio}$ and $\zeta_{relax}$ denote the optimal objective values of \eqref{eq:BigMPenalty} and \eqref{eq:BigMPenaltyRelaxation}, respectively. Unfortunately, \emph{regardless of the data}, \eqref{eq:BigMPenaltyRelaxation} results in almost trivial relaxation gaps. Indeed, note that the solution $\bm{w}=\bm{0}$, $\bm{z}=\bm{1}/M$ is always feasible with an objective value of $(\lambda n)/M\xrightarrow{M\to\infty}0$. Thus, if a large value of $M$ is used, the gap is close to $1$. Additionally, the solutions $(\bm{w},\bm{z})$ obtained from solving the relaxation are uninformative, which will severely impair branch-and-bound algorithms.  \citet{brooks2011support} reports that in the pool of 1,425 instances tested in his paper, while 39\% are solved to optimality, the average gap on the remaining 61\% of instances is over 70\% after 10 minutes of branch-and-bound. Our computations with this formulation using the commercial solver Gurobi, presented in \S\ref{sec:computations}, observed a similar lackluster performance: several instances with $n=100$ cannot be solved in 10 minutes with gaps close to 100\%, and no instance with $n=200$ could be solved to optimality. While off-the-shelf solvers have certainly improved over the past decade in nonlinear MIO problems, these results suggest that improvements have not translated to significantly better methods for problem \eqref{eq:BigMPenalty}. 

\paragraph{Big-M and Hinge loss} The big-M formulation proposed by \citet{brooks2011support} is connected with the popular SVM with hinge loss \eqref{eq:hinge}.
Indeed, note that in optimal solutions of the relaxation \eqref{eq:BigMPenaltyRelaxation} we find that $M z_i^*=\left(1-y_i\left(\bm{a_i^\top w}\right)\right)_+$, and thus (if $M$ is big enough), problem \eqref{eq:BigMPenaltyRelaxation} is equivalent to \eqref{eq:hinge} with $\bar \lambda=\lambda/M$.
In other words, SVMs with the hinge loss \eqref{eq:hinge} can be interpreted as an approximation\footnote{Problem \eqref{eq:hinge} is only a relaxation if the regularization controlling the hinge loss is negligible. Since practical uses of \eqref{eq:BigMPenalty} use much larger parameters, i.e., values of $M$ too small to guarantee the correctness of the MIO formulation, we prefer the term ``approximation'' instead.} of \eqref{eq:BigMPenalty} informed by its natural continuous relaxation.

\paragraph{Cardinality constrained version} Given $k\in \mathbb{Z}_+$, an alternative version of SVMs with 0--1 loss is the cardinality constrained version 
\begin{subequations}\label{eq:BigMCardinality}
	\begin{align}
		\min_{\bm{w},\,\bm z}\;&\|\bm{w}\|_2^2\label{eq:BigMCardinality_obj}\\
		\text{s.t.}\:&y_i\left(\bm{a_i^\top w}\right)\geq 1-Mz_i\quad \forall i\in [n],\label{eq:BigMCardinality_classification}\\
        &\sum_{i=1}^n z_i\leq k\label{eq:BigMCardinality_number}\\
		&\bm{w}\in \R^p,\;\bm{z}\in \{0,1\}^n\label{eq:BigMCardinality_bounds}
	\end{align}
\end{subequations}
where, instead of penalizing missclassified points, constraint \eqref{eq:BigMCardinality_number} restricts the number of missclassifications. The main advantage of using \eqref{eq:BigMCardinality} is when performing hyperparameter selection, e.g., when using cross-validation. Indeed, in \eqref{eq:BigMPenalty}, there are no ``natural" values of $\lambda$ to test. An exhaustive search for an ideal value might be computationally prohibitive, whereas a superficial search might lead to misspecified parameters and lackluster estimators. Instead, for hyperparameter selection in \eqref{eq:BigMCardinality}, it suffices to test $k\in \{0,1,2,\dots,\lfloor n/2\rfloor\}$. Thus, it is possible to select an ideal parameter after a finite number of trials, or have informed a priori guesses for reasonable values for the parameter.

\subsection{Reformulation in an extended formulation}\label{sec:SVMReformExtended}

Observe that problem \eqref{eq:BigMPenalty} can be reformulated with the introduction of additional variables $\bm{x}\in \R^n$, representing the predictions for each datapoint, as 
\begin{subequations}\label{eq:BigMPenalty2}
	\begin{align}
		\min_{\bm{w},\bm{x},\bm z}\;&\|\bm{w}\|_2^2+\lambda\sum_{i=1}^n z_i\label{eq:BigMPenalty2_obj}\\
		\text{s.t.}\:&x_i=1-y_i(\bm{a_i^\top w})\quad \forall i\in [n]\label{eq:BigMPenalty2_Eq}\\
        &y_i\left(\bm{a_i^\top w}\right)\geq 1-Mz_i\quad \forall i\in [n],\label{eq:BigMPenalty2_classification}\\
		&\bm{w}\in \R^p,\;\bm{z}\in \{0,1\}^n\label{eq:BigMPenalty2_bounds}
	\end{align}
\end{subequations}
If we let $\bm{A}\in \R^{n\times p}$ be the matrix whose rows are given by $y_i\cdot \bm{a_i^\top}$, we can rewrite constraints \eqref{eq:BigMPenalty2_Eq} compactly as $\bm{x}=\bm{1}-\bm{A w}$. Then observe that, in an optimal solution, $\bm{w^*}$ is set to the least square solution of system $\bm{x}=\bm{1}-\bm{A w}$, that is, $\bm{w^*}=\bm{A^\dagger}\left(\bm{1}-\bm{x}\right)$ where $\bm{A^\dagger}$ is the Moore-Penrose pseudoinverse of $\bm{A}$. The quadratic term in the objective \eqref{eq:BigMPenalty2_obj} can then be rewritten as
$\bm{1^\top}(\bm{A^\dagger})^\top \bm{A^\dagger 1}-2\bm{1^\top}(\bm{A^\dagger})^\top\bm{A^\dagger x}+\bm{x^\top}(\bm{A^\dagger})^\top \bm{A^\dagger x}$. Introducing set $\bar{\mathcal{X}}$, we find that we can reformulate \eqref{eq:BigMPenalty2} as 
\begin{subequations}\label{eq:SVMReform0}
	\begin{align}
		\bm{1^\top}(\bm{A^\dagger})^\top \bm{A^\dagger 1}+\min_{\bm{w},\bm{x},\bm{X},\bm{ z}}\;&-2\bm{1^\top}(\bm{A^\dagger})^\top\bm{A^\dagger x}+\langle(\bm{A^\dagger})^\top \bm{A^\dagger},\bm{X}\rangle+\lambda\sum_{i=1}^n z_i\label{eq:SVMReform0_obj}\\
		\text{s.t.}\:
        &\bm{x}=\bm{1}-\bm{Aw}\label{eq:SVMReform0_Constr1}\\
        &(\bm{x},\bm{X},\bm{z})\in \bar{\mathcal X},\;\bm{w}\in \R^p\label{eq:SVMReform0_Constr2}.
		\end{align}
\end{subequations}
In particular, we find that SVM problems with the 0-1 loss are a special case of \eqref{prob:two-sided} with $\mathcal{F}=\R^n\times \{0,1\}^n$.  In this paper, we derive big-M free relaxations of $\clconv(\bar{\mathcal X})$ that yield estimators that perform similarly in practice to solving \eqref{eq:BigMPenalty} or \eqref{eq:BigMCardinality} to optimality, but require only a fraction of the computational time.

\paragraph{On scalability} Describing set  $(\bm{x},\bm{X},\bm{z})\in \bar{\mathcal X}$ requires the $n$-dimensional positive semidefinite constraint $\bm{X}-\bm{xx^\top}\in \mathcal{S}_+^n$. In our experiments (with a laptop), the simple addition of this constraint causes off-the-shelf SDP solvers to fail if $n>200$ due to excessive computational and memory requirements. Unfortunately, several practical settings involve more than 200 datapoints, thus a direct application of convexifications of $\bar{\mathcal{X}}$ based on formulation \eqref{eq:SVMReform0} is impractical. In \S\ref{sec:SVM} we discuss an alternative formulation of the proposed strengthenings for $\bar{\mathcal{X}}$, requiring cones of dimension $p$ instead of $n$. The proposed method scales almost linearly with $n$ (but has a high degree of polynomial dependence on $p$), and can solve problems with $n$ in the thousands of seconds, provided that $p$ is small. The dependence on $p$ reduces the scope of the problems that can be tackled with the relaxations proposed in this paper, but settings where $n$ is large and $p$ is small arise in several relevant settings.

Indeed, in many high-stakes domains, such as healthcare and public policy, data tends to be low dimensional since collecting a large number of features through extensive surveys or medical tests is often infeasible. Moreover, while poor prediction performance can lead to significant economic and social costs, the prevalence of noisy labels undermines both learning and fairness, requiring more robust methods \citep{frenay2013classification}. For instance, in healthcare, diagnostic tests are not perfectly accurate, class labeling can be subjective, and survey data is susceptible to misreporting biases \citep{malossini2006detecting, frenay2013classification, obermeyer2019dissecting}. In standard questions on illegal drug usage, sexual activity, contraception, and even vaccination status, data can be unreliable as participants may misrepresent their information \citep{zhao2009non, cheung2017impact}.  In public policy, the use of proxies can introduce noisy labels correlated with protected attributes, and models predicting crime and violence often rely on underreported data affected by response and recall biases \citep{fogliato2020fairness,guerdan2023ground}. In these contexts, methods that are sensitive to outliers and label noise, such as SVMs with hinge loss, often underperform, further emphasizing the need for more robust approaches.

\section{Background on MIO}\label{sec:litReview}
\paragraph{MIO for machine learning} Due to substantial improvements over the past three decades, mixed-integer \emph{linear} optimization technology has seen widespread application in logistics and operational problems \cite{bixby2007progress,bixby2012brief}. Progress in nonlinear MIO has been more limited, initially motivated by problems in process systems engineering \cite{duran1986outer} and finance \cite{bienstock1996computational}, but a recent stream of research has focused on statistical problems. 

Most of the work to date focuses on problems with feature selection or sparsity, where penalties are imposed for features or regression variables that assume non-zero variables. In particular, least squares regression with sparsity has received substantial attention \cite{atamturk2019rank,atamturk2020safe,bertsimas2015or,bertsimas2016best,bertsimas2020sparse,cozad2014learning,cozad2015combined,gomez2021mixed,hazimeh2020fast,miyashiro2015subset,wilson2017alamo,xie2020scalable}, to the point that specialized branch-and-bound solvers that scale to problems with tens of thousands of decision variables have been developed \cite{hazimeh2022sparse}. Other statistical problems with sparsity have also been investigated, such as sparse logistic regression \cite{deza2022safe,sato2016feature,ustun2019learning}, sparse principal component analysis \cite{dey2018convex,kim2022convexification,li2024exact} and sparse SVMs \cite{guan2009mixed,lee2022mixed,dedieu2021learning}.
However, in most settings (besides sparse least squares regression), the proposed approaches resort to off-the-shelf solvers and struggle to solve problems with hundreds of decision variables to optimality. 

Some mixed-integer optimization approaches have also been proposed for statistical learning problems with outliers. Most of the MIO approaches consider least squares problems \cite{gomez2021outlier,gomez2023outlier,insolia2022simultaneous,zioutas2005deleting,zioutas2009quadratic}, although a couple of works in the literature tackle SVM problems \cite{brooks2011support,jammal2020robust,nguyen2013algorithms} or similar classification problems \cite{ustun2016supersparse}. In  most cases, mixed-integer nonlinear problems arising in the context of outlier detection cannot be solved to optimality with current technology. Indeed, existing MIO approaches do not scale beyond problems with a few dozen datapoints. 

\paragraph{Convexification and solution of MIOs}
Most solution methods or approximations of non-convex optimization problems such as \eqref{prob:two-sided} are based on constructing convex relaxations of the non-convex sets. As mentioned before, there has been significant progress in problems with sparsity, which can be modeled naturally as problems with indicator variables and constraints. 
 More precisely, for the case of convex quadratic optimization, consider set $$\mathcal{Y}_{\bm{Q}}:=\{(\bm{x},\bm{z},t)\in\mathbb{R}^n\times\{0,1\}^n\times\mathbb{R}:t\geq \bm{x}^\top \bm{Q}\bm{x},~x_i(1-z_i)=0,\;i=1,\ldots,n\},$$ where variables $\bm{z}$ encode the logical condition $z_i=0\implies x_i=0$ and can be used to penalize variables assuming non-zero values. In addition to its applications in statistical problems, set $\mathcal{Y}_Q$ also arises in portfolio optimization \cite{bienstock1996computational}, unit commitment \cite{bacci2024new}, and model predictive control \cite{lee2024strong}, among others. In practice, the quadratic equality constraints are linearized using big-M constraints 
$$-Mz_i\leq x_i\leq Mz_i,\; i=1,\dots,n,$$
where $M$ is a sufficiently large number. 

Observe that, on the one hand, $\mathcal{Y}_{\bm{Q}}$ is a restriction of $\mathcal{X}_{\bm{Q}}^-$, obtained by replacing constraints $x_i(1-z_i)\leq 0$ with equalities, and in fact $\mathcal{Y}_{\bm{Q}}$ can be represented as the intersection of two sets similar to $\mathcal{X}_{\bm{Q}}^-$. On the other hand, set $\mathcal{X}_{\bm{Q}}$ can also be expressed in an extended formulation via indicator variables as 
\begin{align}
    \mathcal{X}_{\bm{Q}}:=\Big\{(\bm{x},\bm{z},t)\in\mathbb{R}^n\times\{0,1\}^n\times\mathbb{R}:
    &\exists \bm{x^-,\bm{x^+}}\in \R_+^n \text{ such that } t\geq \bm{x}^\top \bm{Q}\bm{x},\notag\\ &\bm{x}=\bm{x^+}-\bm{x^-},\; ~x_i^-z_i= 0,~x_i^+(1-z_i)= 0, ~i=1,\ldots,n\Big\}.\label{eq:extendedReformulation}
\end{align}
Given these links, it should be possible to ``translate" valid inequalities for either $\clconv(\mathcal{X}_{\bm{Q}})$ or $\clconv(\mathcal{Y}_{\bm{Q}})$ to the other set, although how to do so has not been studied in the literature, nor have such relaxations been implemented in practice. Moreover, it is unclear whether there is a direct correspondence between \emph{ideal} relaxations of the two sets. Indeed, the convex hull of an intersection is often a strict subset of the intersection of individual convex hulls, and transformations such as \eqref{eq:extendedReformulation} include additional nonnegative and equality constraints. 


Two special cases of set $\mathcal{Y}_{\bm{Q}}$ are particularly relevant for this paper, corresponding to the cases where $n=1$ and where $\bm{Q}$ is rank-one, formally defined as 
\begin{align}
\mathcal{Y}_{1}&:=\left\{(x,z,t)\in\mathbb{R}\times\{0,1\}\times\mathbb{R}:t\geq x^2,~x(1-z)=0\right\}\\
\mathcal{Y}_{R1}&:=\left\{(\bm{x},\bm{z},t)\in\mathbb{R}^n\times\{0,1\}^n\times\mathbb{R}:t\geq \left(\bm{d^\top x}\right)^2,~x_i(1-z_i)=0,\;i=1,\ldots,n\right\}.
\end{align}
Ideal relaxations for $\mathcal{Y}_1$ were first proposed in \cite{Ceria1999,Frangioni2006} and further explored in \cite{akturk2009strong,Gunluk2010}, while ideal relaxations for $\mathcal{Y}_{R1}$ were originally proposed in \cite{atamturk2019rank}. We summarize these findings in Proposition~\ref{prop:indicatorHull}. 
\begin{proposition}[\citet{atamturk2019rank,Frangioni2006}]\label{prop:indicatorHull}
If $d_i\neq 0$ for all $i=1,\dots,n$, then the identities
\begin{align}
\cl\;\conv (\mathcal{Y}_{1})&:=\left\{(x,z,t)\in\mathbb{R}\times[0,1]\times\mathbb{R}:t\geq x^2/z\right\}\label{eq:indPerspective}\\
\cl\;\conv (\mathcal{Y}_{R1})&:=\left\{(\bm{x},\bm{z},t)\in\mathbb{R}^n\times[0,1]^n\times\mathbb{R}:t\geq \left(\bm{d^\top x}\right)^2/\min\{\bm{1^\top z},1\}\right\}\label{eq:indR1}
\end{align}
hold.
\end{proposition}
More sophisticated relaxations for special cases of matrices $\bm{Q}$ are also available in the literature \cite{wei2024convex,liu2023graph,gomez2024real,lee2024convexification,liu2024polyhedral}, as well as specialized algorithms \cite{das2008algorithms,bhathena2024parametric,bienstock2024solving}. Nonetheless,
practical approaches for general quadratic optimization problems with indicators are still based on perspective relaxations \eqref{eq:indPerspective}. The best performing approaches, however, do not directly resort to off-the-shelf solvers, but rather use additional tailored techniques to effectively handle the nonlinear functions in branch-and-bound \cite{frangioni2016approximated,hazimeh2022sparse}. These approaches are also often coupled with decomposition approaches to identify separable or rank-one terms. Initially, the choice of matrices was obtained directly from the problem statement or via simple heuristics based on the minimum eigenvalues \cite{Frangioni2006}, but later improved heuristics \cite{Frangioni2007} or ``optimal'' approaches resulting in decompositions with the best bound \cite{dong2015regularization,zheng2014improving,frangioni2018decompositions} were proposed.

Variants of set $\mathcal{Y}_{\bm{Q}}$ have also been considered in the literature. In particular, \cite{atamturk2019rank,atamturk2018strong,han20232,atamturk2023supermodularity,wei2020convexification,xie2020scalable,wei2021ideal} study restrictions of $\mathcal{Y}_{\bm{Q}}$ where continuous variables are restricted to be non-negative. The results in \cite{atamturk2023supermodularity} are especially relevant, as they show that ideal descriptions of rank-one sets with non-negative variables are substantially more complicated to describe than \eqref{eq:indR1}. Note that representation \eqref{eq:extendedReformulation} also introduces non-negative variables, thus descriptions of $\text{cl conv}(\mathcal{X}_{\bm{Q}})$ for rank-one matrices could in principle be difficult to represent in the original space of variables. Further generalizations of these results to bounded variables or more general nonlinear sets can be found in \cite{han2024compact,shafiee2024constrained,anstreicher2021quadratic,de2024explicit}.

Given the growing literature on problems with indicator variables involving set $\mathcal{Y}_{\bm{Q}}$, the dearth of results concerning set $\mathcal{X}_{\bm{Q}}$ in the mixed-integer literature is surprising. The non-convex constraints \eqref{eq:two-sided_sign} arise frequently in chance-constrained optimization, often in the form of the big-M constraints \eqref{eq:mixing}, where the binary variables are used to track which scenarios satisfy the constraints. There is certainly extensive literature studying this class of problems \cite{gunluk2001mixing,kuccukyavuz2012mixing,kilincc2019joint,xie2018quantile,abdi2016mixing,luedtke2010integer,zhao2017polyhedral,qiu2014covering,song2014chance}, although the sets under study invariably involve additional constraints such as non-negativity or bounds on the continuous invariables. Indeed, letting $$\mathcal{X}_{\emptyset}:=\{(\bm{x},\bm{z})\in\mathbb{R}^n\times\{0,1\}^n:x_iz_i\geq 0,~x_i(1-z_i)\leq 0, ~i=1,\ldots,n\}$$ be the set without any nonlinear term or side constraints, it can be shown easily that $\clconv(\mathcal{X}_{\emptyset})=\R^n\times [0,1]^n$ (e.g., this result can be obtained as a direct corollary of Theorem~\ref{theo:hullGen} in this paper). In other words, unless the nonlinear term is accounted for in convexifications (or additional constraints are introduced), it is not possible to construct strong relaxations of $\mathcal{X}_{\bm{Q}}$. Thus, existing results from the study of stochastic programs or other mixed-integer linear optimization problems is not applicable to the setting considered in this paper.



\section{Convexification of the rank-one case}\label{sec:hull}

In this section, we derive the convex hull representation for $\mathcal{X}_{\bm{Q}}$ for the case that $\bm{Q}$ is a rank-one matrix, that is, $\bm{Q}=\bm{d}\bm{d}^\top$ for some $\bm{d}\in\mathbb{R}^n$. 
For the ease of exposition, we redefine the rank-one version of $\mathcal{X}_{\bm{Q}}$ as
\begin{equation*}
    \mathcal{X}_{\bm{d}}=\{(\bm{x},\bm{z},t)\in\mathbb{R}^n\times\{0,1\}^n\times\mathbb{R}:t\geq (\bm{d}^\top \bm{x})^2,~ x_i z_i\geq 0,~ x_i(1-z_i)\leq 0,~i=1,\ldots,n\}.
\end{equation*}
Moreover, given $\bm{d}\in \mathbb{R}_+^n$, define the convex set
\begin{equation*} 
\mathcal{\widehat X}_{\bm{d}}=\Cbracket{(\bm{x},\bm{z},t)\in\mathbb{R}^n\times[0,1]^n\times\mathbb{R}:t\geq\frac{(\bm{d}^\top \bm{x})_+^2}{\min\{1,\sum_{i\in [n]} z_i\}} + \frac{(\bm{d}^\top \bm{x})_-^2}{\min\{1,\sum_{i\in [n]} (1-z_i)\}}}.
\end{equation*}

\begin{proposition}[Validity]\label{prop:valid}
For $\bm{d}\geq \bm{0}$, if $(\bm{x},\bm{z},t)\in \mathcal{X}_{\bm{d}}$, then $(\bm{x},\bm{z},t)\in \mathcal{\widehat X}_{\bm{d}}$.
\end{proposition}

\begin{proof}
Consider an arbitrary point $(\bm{x},\bm{z},t)\in \mathcal{X}_{\bm{d}}$. If both denominators are $1$, that is, $\sum_{i\in [n]} z_i \geq 1$ and $\sum_{i\in [n]} (1-z_i)\geq 1$, then we have $t\geq (\bm{d}^\top \bm{x})^2_+ + (\bm{d}^\top \bm{x})^2_-=(\bm{d}^\top \bm{x})^2$. If $\sum_{i\in [n]} z_i=0\Leftrightarrow \bm{z}=\bm{0}$, then $\bm{x}\leq \bm{0}$ and $\bm{d^\top x}\leq 0$, thus $(\bm{d}^\top\bm{x})_+=0$ (the first term vanishes) and $\sum_{i\in [n]} (1-z_i)\geq 1$. In this case, the inequality simplifies to $t\geq (\bm{d}^\top \bm{x})^2_-=(\bm{d}^\top \bm{x})^2$. Similar arguments hold for the case $\sum_{i\in [n]} (1-z_i)=0\Leftrightarrow \bm{z}=\bm{1}$.
\end{proof}

Observe that set $\mathcal{\widehat{X}}_{\bm{d}}$ is indeed convex since it can be represented by SOCP constraints using additional variables. Note first that the inequality describing $\mathcal{\widehat{X}}_{\bm{d}}$ can be expressed by inequalities 
\begin{align}
\label{eq:socp_rep}
t\cdot \min\left\{1,\sum_{i\in [n]}z_i\right\}\geq (\bm{d}^\top \bm{x})^2_+  \text{  and }  t\cdot\min\left\{1,\sum_{i\in [n]}(1-z_i)\right\}\geq (\bm{d}^\top \bm{x})^2_-.
\end{align}
Letting additional variables $w_+:=(\bm{d}^\top \bm{x})_+$, $w_-:=(\bm{d}^\top \bm{x})_-$, $r_+:=\min\{1,\sum_{i\in [n]}z_i\}$ and $r_-:=\min\{1,\sum_{i\in [n]}(1-z_i)\}$, inequalities \eqref{eq:socp_rep} can be formulated as the system
\begin{align*}
&\bm{d}^\top \bm{x} = w_+ + w_- \\
& w_+ \geq 0,~w_-\leq 0 \\
&t \cdot r_+ \geq w^2_+,\; t \cdot r_- \geq w^2_-\\
&1\geq r_+,~ \sum_{i\in [n]}z_i  \geq r_+,~r_+\geq 0 \\
&1\geq r_-,~ \sum_{i\in [n]}(1-z_i)  \geq r_-,~r_-\geq 0.
\end{align*}

The following proposition further shows that the set $\mathcal{\widehat{X}}_{\bm{d}}$ is an ideal convex relaxation for $\mathcal{X}_{\bm{d}}$ when $\bm{d}$ is positive. Then, in Theorem~\ref{theo:hullGen}, we generalize the result to arbitrary vectors $\bm{d}$.

\begin{proposition}\label{propo:rank1} If $d_i> 0$ for all $i\in [n]$, then 
$\mathcal{\widehat{X}}_{\bm{d}}=\textup{cl }\conv(\mathcal{X}_{\bm{d}})$.
\end{proposition}

\begin{proof}
    Consider the two optimization problems: the first one is the mixed-integer formulation
    \begin{align}\label{prob:mip}
        \min\; \bm{\alpha}^\top \bm{x} + \bm{\beta}^\top \bm{z} +\gamma t \text{ s.t. } (\bm{x},\bm{z},t)\in \mathcal{X}_{\bm{d}}
    \end{align}
    and the second is its convex relaxation
   \begin{equation}\label{prob:relax}
        \min \bm{\alpha}^\top \bm{x} + \bm{\beta}^\top \bm{z} +\gamma t \text{ s.t. } (\bm{x},\bm{z},t)\in \mathcal{\widehat{X}}_{\bm{d}},
    \end{equation}
    where $\bm{\alpha},\bm{\beta}\in\mathbb{R}^n,\gamma\in\mathbb{R}$. We show that for any coefficients $\bm{\alpha},\bm{\beta},\gamma$, the two problems share the same optimal value and solution.
   
    We start by discussing all cases in which both \eqref{prob:mip} and \eqref{prob:relax} are unbounded or trivial. First, note that if $\gamma<0$, then both problems are unbounded by letting $t\to\infty$. Second, if $\gamma=0$ and $\alpha_i\neq 0$ for some $i\in [n]$, then both problems are unbounded by either letting $x_i\to\infty$ and setting $z_i=1$, or by letting $x_i\to-\infty$ or $z_i=0$, with all remaining variables set to zero. Third, if $\gamma=0$ and $\bm{\alpha}=\bm{0}$, then letting $z_i=\mathbbm{1}_{\{\beta_i<0\}}$ is optimal for both problems. Thus, it suffices to consider the case $\gamma>0$, and we may assume $\gamma=1$ without loss of generality. By change of variables, we can also assume that $\bm{d}=\bm{1}$.
    Note that if $\alpha_i\neq \alpha_j$ for some $i,j\in [n]$ (without loss of generality, we may assume $\alpha_i<\alpha_j$), then setting $z_i=1$ and $z_j=0$ and $x_i=-x_j=\lambda>0$, with all remaining variables fixed to zero, is feasible for both problems; moreover, both problems are unbounded by letting $\lambda\to\infty$.

    The only remaining case to consider is thus $\gamma=1$, $\bm{d}=\bm{1}$, and $\bm{\alpha}=\eta\bm{1}$ for some $\eta\in \mathbb{R}$. Then, problem (\ref{prob:relax}) can be rewritten as
    \begin{equation}\label{prob:relax-r}
        \min~ \eta \bm{1}^\top \bm{x} + \bm{\beta}^\top \bm{z} + t \text{ s.t. } (\bm{x},\bm{z},t)\in \mathcal{\widehat X}_{\bm{d}},
    \end{equation}
    and we show that there exists an optimal solution of \eqref{prob:relax-r} that is feasible for \eqref{prob:mip} with the same objective value, thus concluding the proof.
    Let $y=\bm{1}^\top \bm{x}$ and define $f:[0,1]^n \rightarrow \mathbb{R}$ as 
    \begin{eqnarray*}
        f(\bm{z}) & := & \bm{\beta}^\top \bm{z} + \min_{(y)_+,(y)_-} \eta ((y)_+ + (y)_-) +  \frac{(y)_+^2}{\min\{1,\sum_{i=1}^n z_i\}} + \frac{(y)_-^2}{\min\{1,\sum_{i=1}^n (1-z_i)\}}.
    \end{eqnarray*}
    Observe that, after projecting out variables $(y)_+, (y)_-$ (equivalently, variables $\bm{x}$), problem (\ref{prob:relax-r}) is equivalent to 
\begin{equation}\label{prob:para-relax}
        \min~  f(\bm{z}) \text{ s.t. } \bm{z}\in[0,1]^n.
    \end{equation}   

    Suppose now that $\eta \leq 0$. For each $\bm{z}\in[0,1]^n$, it is straightforward to verify that the function $f(\bm{z})$  attains its optimal value at $y^*(z) := -\frac{\eta}{2}\min\{1,\sum_{i=1}^n z_i\}$ by differentiating and setting the derivative to zero in the minimization that defines $f(\bm{z})$. Therefore,

        \begin{align*}
        f(\bm{z}) = & \;\bm{\beta}^\top \bm{z} + \min_{y\geq 0} \eta y + \frac{y^2}{\min\{1,\sum_{i=1}^n z_i\}}\\
         = & \;\bm{\beta}^\top \bm{z} - \dfrac{\eta^2}{4} \cdot \min\{1,\textstyle\sum_{i=1}^n z_i\}.
        \end{align*}

    In other words, an optimal solution of \eqref{prob:para-relax} can be obtained by solving the optimization problems 
    \begin{align*}
        \min_{\bm{z}\in [0,1]^n}\;& \bm{\beta}^\top \bm{z} - \frac{\eta^2}{4} \text{ s.t. } \sum_{i=1}^n z_i\geq 1,\text{ and }\\
    \min_{\bm{z}\in [0,1]^n}\;& \bm{\beta}^\top \bm{z} - \frac{\eta^2}{4}\sum_{i=1}^nz_i \text{ s.t. } \sum_{i=1}^n z_i\leq 1,
    \end{align*}
    and selecting the one with the best objective value. Since both of them have totally unimodular constraints and linear objectives, both have integer optimal solutions, and thus \eqref{prob:para-relax} does as well. 
        
    Let $\bar{\bm{z}}$ be a binary optimal solution to \eqref{prob:para-relax}. If $\bar{\bm{z}}=\bm{0}$, then we can set $\bar{\bm{x}}=\bm{0}$ and $t=0$. If $\bar{z}_j=1$ for some $j$, we can set $\bar{x}_j=y^*(\bar{\bm{z}})$ and $\bar{x}_i=0$ for $i\neq j$. By letting $\bar{t}= (\bm{d}^\top\bar{\bm{x}})^2$, the solution $(\bar{\bm{x}},\bar{\bm{z}},\bar{t})$ is optimal for problem (\ref{prob:relax}) and is also feasible for problem (\ref{prob:mip}). Hence, we have shown that for the case $\eta<0$, the original MIP problem (\ref{prob:mip}) and its convex relaxation (\ref{prob:relax}) share the same optimal value. A similar analysis applies to the case $\eta > 0$.
    \end{proof}



We point out that if $d_i=0$ in Proposition~\ref{propo:rank1}, then $\mathcal{\widehat X}_{\bm{d}}$ is a valid but not an ideal relaxation, as stronger relaxations can be obtained by removing variables $z_i$ from the inequalities.
We now generalize Proposition~\ref{propo:rank1} to vectors with arbitrary sign patterns. 

\begin{theorem}\label{theo:hullGen}Given any $\bm{d}\in \R^n$, set $\cl\;\conv (\mathcal{X}_{\bm{d}})$ is described by bound constraints $\bm{0}\leq \bm{z}\leq \bm{1}$ and the inequality
\small\begin{align*}
t\geq\frac{(\bm{d}^\top \bm{x})_+^2}{\min\left\{1,\displaystyle\sum_{i\in \supp_+(\bm{d})} z_i+\sum_{i\in \supp_-(\bm{d})} (1-z_i)\right\}} + \frac{(\bm{d}^\top \bm{x})_-^2}{\displaystyle\min\left\{1,\sum_{i\in \supp_-(\bm{d})}z_i+\sum_{i\in \supp_+(\bm{d})} (1-z_i)\right\}}.\end{align*}\normalsize
\end{theorem}
\begin{proof}
If $\bm{d}>\bm{0}$, the result coincides with Proposition~\ref{propo:rank1}, and if $\bm{d}\geq \bm{0}$ then the result holds since taking sums over the support of $\bm{d}$ removes the unused components of $\bm{z}$. For general $\bm{d}$, consider the change of variables: $x_i= -x'_i$, $z_i= 1-z'_i$  for all $i\in \supp_-(\bm{d})$, and change the signs of $d_i$ accordingly to recover an equivalent problem with $\bm{d}\geq \bm{0}$. We obtain the result by applying Proposition~\ref{propo:rank1} and changing back the variables to their original definitions.
\end{proof}

\paragraph{Connections with earlier works} As pointed out in \S\ref{sec:litReview}, the set $\mathcal{X}_{\bm{Q}}$ is connected to sets with indicator variables previously studied in the literature. Remark~\ref{rem:1dim} below discusses the special case where $n=1$, where the results of Theorem~\ref{theo:hullGen} can indeed be interpreted as a direct application of the perspective reformulation in an extended formulation. However, in Remark~\ref{rem:2+dim} we show that, in general, the descriptions of the convex hull are advantageous compared with existing alternatives.

\begin{remark}[Case $n=1$]
\label{rem:1dim}
Define set $\mathcal{X}_{1}=\{(x,z,t)\in\mathbb{R}\times\{0,1\}\times\mathbb{R}: t\geq x^2,~xz\geq 0,~x(1-z)\leq 0\}.$  
Using the extended formulation in \eqref{eq:extendedReformulation}, set $\mathcal{X}_{1}$ can be rewritten as \small$$\mathcal{X}_1=\{(x,z,t)\in\mathbb{R}\times\{0,1\}\times\mathbb{R}:\exists x^+, x^-\in \R_+ \text{ such that } t\geq (x^+-x^-)^2,~x^-z= 0,~x^+(1-z)= 0, x=x^+-x^-\}.$$\normalsize
Then noting that for any feasible point $(x^+-x^-)^2=(x^+)^2+(x^-)^2$ and using the perspective reformulation \eqref{eq:indPerspective}, we obtain the valid inequality 
$t\geq \frac{(x)_+^2}{z}+\frac{(x)_-^2}{1-z}$ and $0\leq z\leq 1$. Theorem~\ref{theo:hullGen} states that this valid inequality is sufficient to describe the convex hull.
\end{remark}

\begin{remark}[General case]
\label{rem:2+dim}
 Using the extended formulation in \eqref{eq:extendedReformulation}, set $\mathcal{X}_{\bm{d}}$ can be rewritten as \small\begin{align*}
    \mathcal{X}_{\bm{d}}:=\Big\{(\bm{x},\bm{z},t)\in\mathbb{R}^n\times\{0,1\}^n\times\mathbb{R}:
    &\exists \bm{x^-,\bm{x^+}}\in \R_+^n \text{ such that } t\geq \left(\bm{d^\top x^+}-\bm{d^\top x^-}\right)^2,\notag\\ &\bm{x}=\bm{x^+}-\bm{x^-},\; ~x_i^-z_i= 0,~x_i^+(1-z_i)= 0, ~i=1,\ldots,n\Big\}.
\end{align*}\normalsize
In other words, set $\mathcal{X}_{\bm{d}}$ can be written as the epigraph of a rank-one quadratic function with indicator variables and nonnegative constraints, a class of sets that has been studied in the literature. Nonetheless, the resulting strengthening is substantially more complicated and potentially weaker as well:
\begin{itemize}
\item  The strengthening described in proposition 7 of \citet{atamturk2023supermodularity}, in a space of variables including $\bm{x^+}$ and $\bm{x^-}$ but no additional variables, is described as a piecewise conic quadratic function with a factorial number of pieces ($(2n)!+1$) that cannot be directly implemented with off-the-shelf solvers.
\item The conic quadratic strengthenings described in proposition 3 of \citet{han2024compact} or theorem 5 of \citet{shafiee2024constrained} require introducing additional $4n$ variables, in addition to $\bm{x^+}$ and $\bm{x^-}$, and $\mathcal{O}(n)$ constraints. 
\end{itemize}
We see that the formulation in Theorem~\ref{theo:hullGen} is simpler, as it does not require additional variables (including $\bm{x^+}$ and $\bm{x^-}$). 
\end{remark}

\paragraph{The one-sided relaxation}
We close this section by considering the one-sided relaxation \eqref{eq:oneSidedFirst} with $\bm{Q}$ rank-one, that is,
\begin{equation*}
    \mathcal{X}_{\bm{d}}^-=\{(\bm{x},\bm{z},t)\in\mathbb{R}^n\times\{0,1\}^n\times\mathbb{R}:t\geq (\bm{d}^\top \bm{x})^2,~ x_i(1-z_i)\leq 0,~i=1,\ldots,n\}.
\end{equation*}
As Proposition~\ref{corr:oneSided} below states, the only non-trivial case is obtained when all components of $\bm{d}$ have the same sign.
\begin{proposition} \label{corr:oneSided}
\begin{itemize}
    \item [(a)] If $\bm{d} \geq \bm{0}$ (equivalently, if $\bm{d} \leq \bm{0}$), then
$$\cl\; \conv(\mathcal{X}_{\bm{d}}^-)=\Cbracket{(\bm{x},\bm{z},t)\in\mathbb{R}^n\times[0,1]^n\times\mathbb{R}:t\geq\frac{(\bm{d}^\top \bm{x})_+^2}{\min\
\Cbracket{1,\sum_{i\in \supp(\bm{d})}z_i }} + (\bm{d}^\top \bm{x})_-^2}.$$
    \item [(b)] Otherwise, if $d_id_j<0$ for some indexes $i,j\in [n]$, then $$\cl\; \conv(\mathcal{X}_{\bm{d}}^-)=\Cbracket{(\bm{x},\bm{z},t)\in\mathbb{R}^n\times[0,1]^n\times\mathbb{R}:t\geq(\bm{d^\top x})^2}.$$

\end{itemize}
\end{proposition}

\begin{proof} The proof of \textit{(a)} follows from identical reasoning to the proof of Proposition~\ref{propo:rank1}, and we defer it to Appendix~\ref{sec:proofOneSided}. 
For part $\textit{(b)}$, we assume without loss of generality that $d_1>0$ and $d_2<0$. We first note that since any point $(\bm{x},\bm{z},t)$ with $\bm{x}\leq 0$, $\bm{z}\in \{0,1\}^n$ and $t\geq (\bm{d^\top x})^2$ is an element of $\mathcal{X}_{\bm{d}}^-$, we can easily conclude that points satisfying $\bm{ x}\leq 0$, $\bm{0}\leq \bm{ z}\leq \bm{1}$ and $t\geq (\bm{d^\top x})^2$ belong to $\cl\; \conv (\mathcal{X}_{\bm{d}}^-)$. 

Then, observe that for any point $(\bm{\bar x},\bm{\bar z},t)$ with $\bar t\geq (\bm{d^\top \bar x})^2$, $\bm{0}\leq \bm{z}\leq \bm{1}$ and $\bm{d^\top \bar x}\leq 0$, we find that 
$$(\bm{\bar x},\bm{\bar z},\bar t)=\lim_{\lambda\to 0_+}\lambda\underbrace{\left(\frac{\bm{\bar x}}{\lambda}-\left(\frac{\bm{d^\top \bar x}}{d_1\lambda }\right)\bm{e_1},\bm{1},0\right)}_{=(\bm{x^1},\bm{z^1},t^1)}+(1-\lambda)\underbrace{\left(\left(\frac{\bm{d^\top \bar x}}{d_1}\right)\bm{e_1},\bm{\bar z},\bar t\right)}_{=(\bm{x^2},\bm{z^2},t^2)},$$
where $\bm{e_1}$ is first standard basis vector of $\R^n$. Moreover, $(\bm{x^1},\bm{z^1},t^1)\in \mathcal{X}_{\bm{d}}^-$ since $\bm{z^1}=\bm{1}$ imposes no constraints on $\bm{x^1}$ and $\bm{d^\top x^1}=0$; and $(\bm{x^2},\bm{z^2},t^2)\in \text{cl conv}(\mathcal{X}_{\bm{d}}^-)$ since $\bm{x^2}\leq \bm{0}$ and $\bm{0}\leq \bm{z^2}\leq \bm{1}$. Thus, $(\bm{\bar x},\bm{\bar z},\bar t)\in \cl\;\conv(\mathcal{X}_{\bm{d}}^-)$.
If $\bm{d^\top \bar x}>0$, a similar argument can be repeated using $\bm{e_2}$ instead of $\bm{e_1}$.
\end{proof}

\section{Valid inequalities in the extended formulation}\label{sec:validExtended}
While ideal formulations for $\mathcal{X}_{\bm{Q}}$ with general matrices may be difficult to obtain, we now discuss how to construct relaxations based on the rank-one convexification, using set $\bar{\mathcal{X}}$ defined in \eqref{eq:defBarX}. First, we point out that sets $\clconv(\mathcal{X}_{\bm{Q}})$ and $\clconv(\bar{\mathcal{X}})$ are closely connected, see Proposition~\ref{prop:connection} below. Indeed, on the one hand, $\clconv(\mathcal{X}_{\bm{Q}})$ is a projection of $\clconv(\bar{\mathcal{X}})$ in a lower-dimensional space, and
on the other hand, $\clconv(\bar{\mathcal{X}})$ can be interpreted as the intersection of (infinitely many) sets of the form $\clconv(\mathcal{X}_{\bm{Q}})$.

\begin{proposition}\label{prop:connection}
The following two statements hold true:
\begin{enumerate}[label=(\alph*)]
\item Given a point $(\bm{x},\bm{z},t)\in \R^{2n+1}$,  $(\bm{x},\bm{z},t)\in \conv(\mathcal{X}_{\bm{Q}})$ if and only if there exists $\bm{X}\in \R^{n\times n} \text{ such that }t\ge \langle \bm{Q},\bm{X}\rangle \text{ and } (\bm{x},\bm{X},\bm{z})\in \conv(\bar{\mathcal{X}})$.
\item Given a point $(\bm{x},\bm{X},\bm{z})\in \R^{2n+n^2}$, $(\bm{x},\bm{X},\bm{z})\in \clconv(\bar{\mathcal{X}})$ if and only if for all $\bm{Q}\in \mathcal{S}_+^n$ there exists  $t\in \R \text{ such that } t=\langle \bm{Q},\bm{X}\rangle \text{ and } (\bm{x},\bm{z},t)\in \clconv(\mathcal{X}_{\bm{Q}})$.
\end{enumerate}
\end{proposition}
\begin{proof}
To prove part (a), define 
\[ \tilde{\mathcal{X}}=\left\{ (\bm x,\bm X,\bm z,t):\; (\bm x,\bm X,\bm z)\in\bar{\mathcal{X}},\;t\ge\innerProd{\bm Q}{\bm X} \right\}. \]Then one has $\mathcal{X}_{\bm Q}=\proj_{\bm x,\bm z, t}\tilde{\mathcal{X}}$ which represents the projection of $\tilde{\mathcal{X}}$ onto the space of $(\bm x,\bm z,t)$.  We immediately find that 
\begin{align*}
    \conv(\mathcal{X}_{\bm Q})=&\conv\proj\nolimits_{\bm x,\bm z,t}\tilde{\mathcal{X}}\\
    =& \proj\nolimits_{\bm x,\bm z,t}\conv\tilde{\mathcal{X}}\\
    =&\proj\nolimits_{\bm x,\bm z,t}\left\{(\bm x,\bm X, \bm z,t ):\,(\bm x,\bm X,\bm z)\in\conv\left(\bar{\mathcal{X}}\right),\;t\ge\innerProd{\bm Q}{\bm X}   \right\},
\end{align*}
which proves the first statement.

To prove part (b), we consider minimization of a linear function over the two sets involved, that is,
\begin{align}
&\min_{\bm{x},\bm{X},\bm{z}}\; \bm{a^\top x}+\bm{c^\top z}+\langle \bm{\Sigma},\bm{X}\rangle \text{  s.t.  }(\bm{x},\bm{X},\bm{z})\in \clconv(\bar{\mathcal{X}}),\text{ and}\label{eq:opt1}\\
&\min_{\bm{x},\bm{X},\bm{z}}\; \bm{a^\top x}+\bm{c^\top z}+\langle \bm{\Sigma},\bm{X}\rangle \text{  s.t.  }(\bm{x},\bm{z},\langle \bm{Q},\bm{X}\rangle)\in \cl\; \conv(\mathcal{X}_{\bm{Q}})\text{ for all }\bm{Q}\in \mathcal{S}_+^n\label{eq:opt2},
\end{align}
and show that both problems have the same optimal solutions and objective values.

First note that if $\bm{\Sigma}\not\in \mathcal{S}_+^n$, then both problems are unbounded. Indeed, it is feasible for both problems to set $\bm{X}=\bm{xx^\top}+\gamma\bm{vv^\top}$ where $\gamma\geq 0$ and $\bm{v}$ is the eigenvector associated with the minimum eigenvalue of $\bm{\Sigma}$; letting $\gamma\to\infty$, we find feasible solutions with arbitrarily small objective values. 

If $\bm{\Sigma}\in \mathcal{S}_+^n$, then there exists an optimal solution of \eqref{eq:opt1} that is also optimal for 
\begin{align}
\min_{\bm{x}\in \R^n,\bm{z}\in \{0,1\}}\;&\bm{x^\top \Sigma x}+\bm{a^\top x}+\bm{c^\top z} \text{ s.t. }x_iz_i\geq 0,\; x_i(1-z_i)\leq 0,\; \forall i\in [n]\label{eq:equivMip}
\end{align}
with same objective value.
Note that problem \eqref{eq:opt2} is a relaxation of \eqref{eq:equivMip}, since any feasible solution $(\bm{x},\bm{z})$ of \eqref{eq:opt2} yields a feasible solution for \eqref{eq:equivMip} by setting $\bm{X}=\bm{xx^\top}$, with the same objective value. 
Now consider the further relaxation of \eqref{eq:opt2} where constraints for all matrices $\bm{Q}\neq \bm{\Sigma}$ are dropped:
\begin{align*}
&\min_{\bm{x},\bm{X},\bm{z}}\;\langle{\bm{\Sigma},\bm{X}}\rangle+\bm{a^\top x}+\bm{c^\top z} \text{ s.t. }(\bm{x},\bm{z}, \langle \bm{\Sigma},\bm{X}\rangle)\in \clconv(\mathcal{X}_{\bm{\Sigma}})\\
=&\min_{(\bm{x},\bm{z},t)\in \clconv(\mathcal{X}_{\bm{\Sigma}})}\;t+\bm{a^\top x}+\bm{c^\top z}=\min_{(\bm{x},\bm{z},t)\in \mathcal{X}_{\bm{\Sigma}}}\;t+\bm{a^\top x}+\bm{c^\top z},
\end{align*}
which reduces to \eqref{eq:equivMip}.
In conclusion, we find that \eqref{eq:opt2} is equivalent to \eqref{eq:equivMip} and to \eqref{eq:opt1}.
\end{proof}

Thus, a valid relaxation of $\mathcal{\bar X}$ based on the results of $\S\ref{sec:hull}$ can be obtained by relaxing the second condition in Proposition~\ref{prop:connection}, to hold for rank-one matrices only. We summarize this relaxation in the following corollary.

\begin{corollary}[Validity]\label{cor:validGenForm}
    If $(\bm{x},\bm{X},\bm{z})\in \mathcal{\bar X}$, then 
$\left(\bm{x},\bm{z},\langle \bm{dd^\top},\bm{X}\rangle\right)\in \clconv(\mathcal{X}_{\bm{d}})$ for all $\bm{d}\in \R^n$.
\end{corollary} 

In order to simplify the notation, we limit our study of this section to a weaker relaxation, where constraints are imposed for non-negative vectors only, that is
\begin{align}\label{eq:relaxPosOnly}
\left(\bm{x},\bm{z},\langle \bm{dd^\top},\bm{X}\rangle\right)\in \clconv(\mathcal{X}_{\bm{d}})\qquad \forall\bm{d}\in \R_+^n
\end{align}
By focusing on the simpler set \eqref{eq:relaxPosOnly}, we can use the simpler relaxations in Proposition~\ref{prop:valid} instead of the more general but more cumbersome result given in Theorem~\ref{theo:hullGen}. 

A first approach to implement relaxation \eqref{eq:relaxPosOnly} is to add valid inequalities for any fixed vector $\bm{d}$, immediately leading to valid conic quadratic inequalities.
 \begin{proposition}[Valid conic quadratic inequalities]
For any $\bm{d}\geq \bm{0}$, the inequality 
\begin{align}\label{eq:validFixedD}\langle\bm{dd^\top},\bm{X}\rangle \geq \frac{(\bm{d^\top x})^2_+}{\min\{1,\sum_{i\in \supp(\bm{d})} z_i\}} + \frac{(\bm{d^\top x})^2_-}{\min\{1,\sum_{i\in \supp(\bm{d})} (1-z_i)\}}
\end{align}
is valid for $\mathcal{\bar X}$.
 \end{proposition}
 \begin{proof}
Take the set \eqref{eq:relaxPosOnly} with $\bm{d}$ fixed and use the representation of $\clconv(\mathcal{X}_{\bm{d}})$ given in Theorem~\ref{theo:hullGen}.
 \end{proof}
 
Naturally, the quality of the relaxation induced by a single inequality~\eqref{eq:validFixedD} depends critically on the choice of vector $\bm{d}$ and can be weak in practice. As we now show, the infinite set of conic quadratic inequalities \eqref{eq:relaxPosOnly} can be written as a finite number of copositive inequalities in the same space of variables.

 \begin{proposition}\label{prop:validX}
The infinite family of valid inequalities given by $\bm{X}\succeq \bm{xx^\top}$ and 
\begin{align}\langle\bm{dd^\top},\bm{X}\rangle \geq \frac{(\bm{d^\top x})^2_+}{\min\{1,\sum_{i\in [n]} z_i\}} + \frac{(\bm{d^\top x})^2_-}{\min\{1,\sum_{i\in [n]} (1-z_i)\}}\quad \forall \bm{d}\in \R_+^n\label{eq:preCP}\end{align}
can be written compactly as $\bm{X}\succeq \bm{xx^\top}$ and the two copositive constraints 
\begin{align}
     &\begin{pmatrix}
                \sum_{i\in [n]} z_i & -\bm{x^\top}\\
                -\bm{x}                      & \bm{X}
            \end{pmatrix}\in \mathcal{C}_+^{n+1} \text{  and  }
                                  \begin{pmatrix}
                \sum_{i\in [n]} (1-z_i) & \bm{x^\top}\\
                \bm{x}                      & \bm{X}
            \end{pmatrix}\in \mathcal{C}_+^{n+1}.
            \label{eq:copConstraints}
\end{align}
 \end{proposition}
 \begin{proof} 
     Notice that the inequalities in (\ref{eq:preCP}) can be expressed equivalently by\begin{subequations}
    \begin{align}
        & \min\{1,\textstyle\sum_{i\in [n]}z_i\}\cdot\left(\bm{d^\top X d}\right) \geq (\bm{d}^\top \bm{x})^2_+ &  \forall \bm{d}\in \R_+^n,\label{eq:plusIneq}\\
        & \min\{1,\textstyle\sum_{i\in [n]} (1-z_i)\}\cdot\left(\bm{d^\top X d}\right) \geq (\bm{d}^\top \bm{x})^2_- & \forall\bm{d}\in \R_+^n. \label{eq:minusIneq}
    \end{align}
\end{subequations} 
Moreover, (\ref{eq:plusIneq}) can be rewritten as 
\begin{equation}\label{eq:plusPreSchur}
    \begin{array}{ccc}
    \displaystyle\bm{d^\top} \Cparen{\textstyle \min\left\{1,\sum_{i\in [n]} z_i\right\}\bm{X}-\bm{x}\bm{x}^\top}\bm{d}\geq 0\quad \forall \bm{d}\in \R_+^n \text{ such that } \bm{d^\top} (-\bm{x})\leq 0,
    \end{array}
\end{equation}
as $\bm{X}\succeq \bm{0}$ ensures that the inequality is automatically satisfied if $\bm{d^\top x}\leq 0$.
 It follows from the copositive Schur complement \cite{ping1993criteria} that \eqref{eq:plusPreSchur} is equivalent to
\begin{equation}\label{eq:plusWithMin}
    \begin{pmatrix} \min\left\{1,\sum_{i\in [n]} z_i\right\} & -\bm{x^\top} \\ -\bm{x} & \bm{X}\end{pmatrix} \in \mathcal{C}_+^{n+1}.    
\end{equation}
Inequality \eqref{eq:plusWithMin} can be decomposed into $\begin{pmatrix} 1 & -\bm{x^\top} \\ -\bm{x} & \bm{X}\end{pmatrix}\in \mathcal{C}_+^{n+1}$ and $\begin{pmatrix} \sum_{i\in [n]} z_i & -\bm{x^\top} \\ -\bm{x} & \bm{X}\end{pmatrix}\in \mathcal{C}_+^{n+1}$, and the first one can be omitted as it is already implied by inequality $\bm{X}\succeq \bm{xx^\top}$.
With identical arguments, we conclude that \eqref{eq:minusIneq} can be expressed as
    $\begin{pmatrix} \sum_{i\in [n]} (1-z_i) & \bm{x^\top} \\ \bm{x} & \bm{X}\end{pmatrix}\in \mathcal{C}_+^{n+1},$
concluding the proof.
 \end{proof}

Observe that relaxation \eqref{eq:preCP} is weaker than \eqref{eq:relaxPosOnly}, since the sums of the elements of $\bm{z}$ are taken over all indexes in $[n]$ rather than over the support $\supp(\bm{d})$. A more general family of valid inequalities is described in Corollary~\ref{corr:validXSupport} below.

 \begin{corollary}[Valid copositive inequalities]\label{corr:validXSupport}
Given any $S\subseteq [n]$, the two copositive inequalities
\begin{align*}
     &\begin{pmatrix}
                \sum_{i\in S} z_i & -\bm{x_S^\top}\\
                -\bm{x_S}                      & \bm{X_S}
            \end{pmatrix}\in \mathcal{C}_+^{|S|+1} \text{  and  }
                                  \begin{pmatrix}
                \sum_{i\in S} (1-z_i) & \bm{x_S^\top}\\
                \bm{x_S}                      & \bm{X_S}
            \end{pmatrix}\in \mathcal{C}_+^{|S|+1} 
\end{align*}
are valid for $\mathcal{\bar X}$. Moreover, constraints \eqref{eq:relaxPosOnly} are equivalent to imposing the copositive inequalities for all sets $S\subseteq [n]$.
 \end{corollary}


Copositive programs are notoriously hard to solve, calling into question how practical Proposition~\ref{prop:validX} and Corollary~\ref{corr:validXSupport} are. As we show next, it is possible to rewrite the copositive constraints as positive semi-definite constraints in an extended formulation, by exploiting the fact that $\bm{X}\succeq \bm{xx^\top}$ holds.

\begin{proposition}\label{propo:cop-sdp}
The system
\small\begin{align}\label{eq:psdEquiv}
\exists \bm{g},\bm{h}\in \R^n \text{ such that }\bm{h}\leq \bm{x}\leq \bm{g},\;
\begin{pmatrix}
\sum_{i\in [n]} z_i & -\bm{g^\top}\\
-\bm{g}& \bm{X}
\end{pmatrix}\in \mathcal{S}_+^{n+1},\; \text{and }\begin{pmatrix}
\sum_{i\in [n]} (1-z_i) & \bm{h^\top}\\
\bm{h}& \bm{X}
\end{pmatrix}\in \mathcal{S}_+^{n+1}
\end{align}\normalsize

\noindent is valid for $\bar{\mathcal X}$ and is equivalent to imposing constraints \eqref{eq:copConstraints}.
\end{proposition}
The proof of a (slightly) more general result can be found in Appendix~\ref{sec:proofEquiv}. We point out that validity of the inequalities can be easily verified by cases. For constraint $\begin{pmatrix}
\sum_{i\in [n]} z_i & -\bm{g^\top}\\
-\bm{g}& \bm{X}
\end{pmatrix}\in \mathcal{S}_+^{n+1}$, note that if $\sum_{i\in [n]}z_i\geq 1$, then we may set $\bm{g}=\bm{x}$ to satisfy the constraint, as indeed in that case $$\begin{pmatrix}
\sum_{i\in [n]} z_i & -\bm{x^\top}\\
-\bm{x}& \bm{X}
\end{pmatrix}\succeq \begin{pmatrix}
1 & -\bm{x^\top}\\
-\bm{x}& \bm{X}
\end{pmatrix} \text{ and } \begin{pmatrix}
1 & -\bm{x^\top}\\
-\bm{x}& \bm{X}
\end{pmatrix}\in \mathcal{S}_+^{n+1}\Leftrightarrow \bm{X}\succeq \bm{xx^\top}.$$ Moreover, if $\sum_{i\in [n]}z_i=0$, then $\bm{x}\leq \bm{0}$ for any point in $\bar{\mathcal X}$. Thus we may set $\bm{g}=\bm{0}\geq \bm{x}$, and the constraint $\begin{pmatrix}
0 & \bm{0^\top}\\
\bm{0}& \bm{X}
\end{pmatrix}\succeq \bm{0}\Leftrightarrow \bm{X}\succeq \bm{0}$ is trivially satisfied. The second constraint in \eqref{eq:psdEquiv} can be verified to be valid with identical arguments.

\paragraph{One-sided version}
Consider the one-sided relaxation of $\mathcal{\bar X}$ obtained by dropping the constraints $x_iz_i\geq 0$, that is, \begin{align*}\mathcal{\bar X}^-:=\big\{(\bm{x},\bm{X},\bm{z})\in\mathbb{R}^n\times\R^{n\times n}\times\{0,1\}^n:&\;\bm{X}\succeq \bm{xx^\top},~x_i(1-z_i)\leq 0, \forall i\in [n]\big\}.\end{align*}
From Proposition~\ref{corr:oneSided}, we know that the strengthened  inequalities involve only the positive part $(\bm{d^\top x})_+$. By repeating the arguments used previously in this section, we may derive valid inequalities for $\mathcal{\bar X}^-$, which we condense in the following proposition (stated without proof).

\begin{proposition}\label{prop:oneSidedExtended}
Given any $S\subseteq [n]$, the  inequality $\begin{pmatrix}\sum_{i\in S}z_i&-\bm{x_S^\top}\\-\bm{x_S}&\bm{X_S}\end{pmatrix}\in \mathcal{C}_+^{|S|+1}$ is valid for $\mathcal{\bar X}^-$, and can be equivalently written in an extended formulation as  
\begin{align*}
\exists \bm{g}\in \R^S \text{ such that }\bm{x_S}\leq \bm{g},\;
\begin{pmatrix}
\sum_{i\in S} z_i & -\bm{g^\top}\\
-\bm{g}& \bm{X_S}
\end{pmatrix}\in \mathcal{S}_+^{|S|+1}.
\end{align*}
\end{proposition}

 \begin{remark}[Comparison with sets with indicator variables]
 Consider the set with indicator variables defined as 
\begin{align*}\mathcal{\bar Y}:=\big\{(\bm{x},\bm{X},\bm{z})\in\mathbb{R}^n\times\R^{n\times n}\times\{0,1\}^n:&\;\bm{X}\succeq \bm{xx^\top},~x_i(1-z_i)= 0,\;i=1,\dots,n\big\}.\end{align*}
\citet{atamturk2019rank} show that valid inequalities for this set, derived from rank-one strengthening, are given by conic constraints of the form $\begin{pmatrix}\sum_{i\in S}z_i&\bm{x_S^\top}\\\bm{x_S}&\bm{X_S}\end{pmatrix}\in \mathcal{S}_+^{|S|+1}\Leftrightarrow\begin{pmatrix}\sum_{i\in S}z_i&-\bm{x_S^\top}\\-\bm{x_S}&\bm{X_S}\end{pmatrix}\in \mathcal{S}_+^{|S|+1}$. Proposition~\ref{prop:oneSidedExtended} shows that to obtain similar inequalities for the relaxation $\bar{\mathcal{X}}^-$, it suffices to relax the inequalities from the positive semidefinite cone to the copositive cone (while being careful with the sign of the continuous variables). 
 \end{remark}

\section{Formulations for SVM problems}\label{sec:SVM}
We now apply the theory in \S\ref{sec:hull} and \S\ref{sec:validExtended} to design relaxations to SVM problems with the 0--1 loss \eqref{eq:BigMPenalty} and \eqref{eq:BigMCardinality}. As we show later in our computational section, the solutions from the relaxations can be used directly as estimators for this class of problems, outperforming standard SVM approaches in settings with noise and outliers and requiring only a small fraction of the time to solve the problems to optimality using MIO solvers.

\subsection{Strong rank-one reformulations for SVM problems}
A direct implementation of the proposed relaxation would simply apply the inequalities in Proposition~\ref{prop:oneSidedExtended} to formulation \eqref{eq:SVMReform0}. However, as discussed in \S\ref{sec:SVMReformExtended}, this approach is computationally too expensive due to the difficulty of handling large-dimensional conic constraints such as $\bm{X}-\bm{xx^\top}\in \mathcal{S}_+^{n+1}$. 
The approach we propose is to implement the strengthening with lower-dimensional cones with dimension depending on the number of features $p$, but not on the number of datapoints $n$. 

Toward this end, recall formulation \eqref{eq:BigMPenalty2}, which we repeat for convenience:
\begin{subequations}\label{eq:SVMRepeated}
	\begin{align}
		\min_{\bm{w},\bm{x},\bm z}\;&\|\bm{w}\|_2^2+\lambda\sum_{i=1}^n z_i\\
		\text{s.t.}\:&x_i=1-y_i(\bm{a_i^\top w})\quad \forall i\in [n]\\
		&y_i\left(\bm{a_i^\top w}\right)\geq 1-Mz_i\quad \forall i\in [n],\\
		&\bm{w}\in \R^p,\;\bm{z}\in \{0,1\}^n
	\end{align}
    \end{subequations}
Consider introducing additional variables $\bm{W}\in \R^{p\times p}$ representing products $W_{ij}=w_iw_j$. Observe that $x_ix_j=\left(1-y_i\cdot \bm{a_i^\top w}\right)\left(1-y_j\cdot \bm{a_j^\top w}\right)=1-y_i\cdot \bm{a_i^\top w}-y_j\cdot \bm{a_j^\top w}+y_iy_j\cdot \bm{a_i^\top ww^\top a_j}$; in matrix form, given any subset $S\subseteq [n]$, we can write the equality
\begin{align*}
\bm{x_Sx_S^\top}=&\bm{1_S1_S^\top}-\bm{A_{S,p} w}\bm{1_S^\top}-\bm{1_S}\bm{w^\top A_{S,p}^\top}+\bm{A_{S,p} w_Sw_S^\top A_{S,p}^\top}\\
=&\bm{1_S}\bm{1_S^\top}-\bm{A_{S,p} w}\bm{1_S^\top}-\bm{1_S}\bm{w^\top A_{S,p}^\top}+\bm{A_{S,p} W A_{S,p}^\top},
\end{align*}
where $\bm{A}$ is the matrix whose $i$-th row is $y_i\cdot\bm{a_i}$ and $\bm{A_{S,p}}\in \R^{|S|\times p}$ is the submatrix of $\bm{A}$ containing rows indexed by $S$. Thus, in the context of SVM problems, we may write the valid inequalities of Proposition~\ref{prop:oneSidedExtended} as
\begin{align*}
\exists \bm{g}\in \R^S \text{ such that }\bm{1_S}-\bm{A_{S,p} w}\leq \bm{g},\;
\begin{pmatrix}
\sum_{i\in S} z_i & -\bm{g^\top}\\
-\bm{g}& \bm{1_S}\bm{1_S^\top}-\bm{A_{S,p} w}\bm{1_S^\top}-\bm{1_S}\bm{w^\top A_{S,p}^\top}+\bm{A_{S,p} W A_{S,p}^\top}
\end{pmatrix}\in \mathcal{S}_+^{|S|+1}.
\end{align*}
The validity of this class of inequalities can be checked directly using identical arguments to Proposition~\ref{propo:cop-sdp}. Indeed, we can always set $\bm{W}=\bm{ww^\top}$ and: $\bullet$ if no points in $S$ are misclassified, then $\bm{1_S}-\bm{A_{S,p}w}\leq 0$, and we may set $\bm{z_S}=\bm{g_S}=\bm{0_S}$; $\bullet$ otherwise $\sum_{i\in S}z_i\geq 1$ and we can set $\bm{g}=\bm{1_S}-\bm{A_{S,p} w}$, since 
$$\begin{pmatrix}
	1 & \left(\bm{A_{S,p} w}-\bm{1_S}\right)^\top\\
\bm{A_{S,p} w}-\bm{1_S}& \bm{1_S}\bm{1_S^\top}-\bm{A_{S,p} w}\bm{1_S^\top}-\bm{1_S}\bm{w^\top A_{S,p}^\top}+\bm{A_{S,p} ww^\top A_{S,p}^\top}
\end{pmatrix}\in \mathcal{S}_+^{|S|+1}$$
is always satisfied.

We formalize the formulation proposed for SVM problems in the following proposition.
\begin{proposition}\label{prop:formulation}
Given any collection $\mathcal L\subseteq 2^{[n]}$ of subsets of $[n]$, the formulation
\small\begin{subequations}\label{eq:SVMReform1}
	\begin{align}
		\min_{\substack{\bm{w},\bm{W},\bm{ z}\\\{\bm{g^L}\}_{L\in \mathcal{L}}}}\;&\langle \bm{I}_{p\times p}, \bm{W}\rangle+\lambda\sum_{i=1}^n z_i\label{eq:SVMReform1_obj}\\
		\text{s.t.}\:&\bm{1_L}-\bm{A_{L,p} w}\leq \bm{g^L},\;
\begin{pmatrix}
\sum_{i\in L} z_i & -(\bm{g^L})^\top\\
-\bm{g^L}& \bm{1_L}\bm{1_L^\top}-\bm{A_{L,p} w}\bm{1_L^\top}-\bm{1_L}\bm{w^\top A_{L,p}^\top}+\bm{A_{L,p} W A_{L,p}^\top}
\end{pmatrix}\in \mathcal{S}_+^{|L|+1},\; \forall L\in \mathcal{L}\label{eq:SVMReform1_CP}\\
        &\bm{W}-\bm{ww^\top}\in \mathcal{S}_+^p \label{eq:SVMReform1_W}\\
		&\bm{w}\in \R^p,\;\bm{W}\in \R^{p\times p},\;\bm{z}\in [0,1]^n,\; \{\bm{g^L}\in \R^L\}_{L\in \mathcal{L}}\label{eq:SVMReform1_bounds}
	\end{align}
\end{subequations}\normalsize
is a convex relaxation of \eqref{eq:BigMPenalty}. If constraint $\sum_{i=1}^n z_i\leq k$ is added and term $\lambda\sum_{i=1}^n z_i$ is removed from the objective, the ensuing formulation is a convex relaxation of \eqref{eq:BigMCardinality}.
\end{proposition}

\subsection{Interpretation as decomposition}

We now provide some high-level intuition and interpretation of the formulation proposed in Proposition~\ref{prop:formulation}. Consider an alternative application of the formulations proposed, consisting of using the valid inequalities in Proposition~\ref{corr:oneSided}, without introducing matrix variables $\bm{X}$ or $\bm{W}$.  Then, given any matrix $\bm{F}\in \R_+^{q\times n}$ (for some $q\in \mathbb{Z}_+$), we can rewrite \eqref{eq:SVMRepeated} as
\begin{subequations}\label{eq:SVMReform2}
	\begin{align}
		\min_{\bm{w},\bm{x},\bm{X},\bm{ z}}\;&\|\bm{w}\|_2^2-\left\|\bm{F}\left(\bm{1}-\bm{A w}\right)\right\|_2^2+\|\bm{Fx}\|_2^2+\lambda\sum_{i=1}^n z_i\label{eq:SVMReform2_obj}\\
		\text{s.t.}\:&\bm{x}=\bm{1}-\bm{A w}\label{eq:SVMReform2_Eq}\\
        &x_i(1-z_i)\leq 0\quad \forall i\in [n]\\
		&\bm{w}\in \R^p,\;\bm{x}\in \R^n,\;\bm{z}\in \{0,1\}^n.\label{eq:SVMReform2_bounds}
	\end{align}
\end{subequations}
Then, letting $\bm{f_j}\in \R^n$ denote the $j$-th row of $\bm{F}$, we find from Proposition~\ref{corr:oneSided} that
\begin{align}\label{eq:substitution}
t\geq\|\bm{Fx}\|_2^2&\Leftrightarrow t\geq \sum_{j=1}^q\left(\bm{f_j^\top x}\right)^2\implies t\geq \sum_{j=1}^q\left(\frac{(\bm{f_j^\top x})_+^2}{\min\{1,\sum_{i\in \supp(\bm{f_j})}z_i\}}+(\bm{f_j^\top x})_-^2\right).
\end{align}
Thus, if term $\|\bm{w}\|_2^2-\left\|\bm{F}\left(\bm{1}-\bm{A w}\right)\right\|_2^2$ is convex (that is, $\bm{I}_{p\times p}\succeq \bm{A^\top F^\top FA}$) and term $\|\bm{Fx}\|_2^2$ in the objective \eqref{eq:SVMReform2_obj} is replaced using \eqref{eq:substitution}, we obtain a stronger reformulation of \eqref{eq:SVMRepeated}.
 However, the quality of relaxation depends on $\bm{F}$, which must be chosen in an informed way. As we now show, the formulation proposed in Proposition~\ref{prop:formulation} can be shown to be at least as strong as formulations obtained using representation \eqref{eq:substitution}.

  \begin{proposition}\label{prop:decomposition}
If $\supp(\bm{f_j})\in  \mathcal{L}$ for all $j\in [q]$, then \eqref{eq:SVMReform1} is at least as strong as the formulation obtained from using strengthening \eqref{eq:substitution}.
 \end{proposition}
\begin{proof}
    Our proof technique is based on constructing successive relaxations of \eqref{eq:SVMReform1} and showing that the last relaxation in the sequence is still stronger than using \eqref{eq:substitution}. We begin by reformulating the constraints \eqref{eq:SVMReform1_CP} using the Schur complement for copositive constraints (see
Corollary~\ref{corr:validXSupport}), as the (infinite) set of constraints
    \begin{align}\label{eq:SVMCPReformulation1}\langle\bm{f f^\top},\bm{X} \rangle \geq \frac{(\bm{f^\top x})^2_+}{\min\{1,\sum_{i\in  L} z_i\}} + (\bm{f^\top x})^2_-\quad \forall  L\in \mathcal{L},\; \forall \bm{f}\in \R_+^n:\supp(\bm{f})\subseteq L,
    \end{align}
where $\bm{X} = \bm{1}\bm{1^\top}-\bm{A w}\bm{1^\top}-\bm{1}\bm{w^\top A^\top}+\bm{A W A^\top}$ and $\bm{x} = \bm{1}-\bm{A w}$. Consider a relaxation of the set of constraints \eqref{eq:SVMCPReformulation1} where: \textit{(i)} constraints associated with sets $L\neq \supp(\bm{f_j})$ for some $j\in [q]$ are removed, and \textit{(ii)} if $L=\supp(\bm{f_j})$ for some $j\in [q]$, then constraints associated with any vector $\bm{f}\neq \bm{f_j}$ are also removed. In other words, we obtain the relaxed system
    \begin{align} 0\geq&-\langle\bm{f_j f_j^\top}, \bm{1}\bm{1^\top}-\bm{A w}\bm{1^\top}-\bm{1}\bm{w^\top A^\top}+\bm{A W A^\top} \rangle \notag\\
    &+ \frac{(\bm{f_j^\top (\bm{1}-\bm{A w}))})^2_+}{\min\{1,\sum_{i\in  \supp(\bm{f_j})} z_i\}} + (\bm{f_j^\top (\bm{1}-\bm{A w}))})^2_-\quad \forall j\in [q].\label{eq:SVMCPReformulation2}
    \end{align}
Next, we dualize constraints \eqref{eq:SVMCPReformulation2} by introducting the dual variables $d_j$ for each $j\in [q]$. We obtain the Lagrangian dual
    \begin{align}
        \max_{\bm{d}\in \R_+^n}\min_{\bm{w},\bm{W},\bm{ z}}\;&\left\langle \bm{I}_{p\times p}-\sum_{j=1}^q d_j \bm{A^\top f_j f_j^\top A}, \bm{W}\right\rangle +\left(2d_j \sum_{j=1}^q\bm{A^\top f_j f_j^\top 1}\right)^\top\bm{w} - \sum_{j=1}^qd_j \bm{ 1^\top f_j f_j^\top 1} \notag\\ 
        &+\sum_{j=1}^q d_j\frac{(\bm{f_j^\top}(\bm{1}-\bm{A w}))^2_+}{\min\{1,\sum_{i\in  \supp(\bm{f_j})} z_i\}} + \sum_{j=1}^qd_j (\bm{f_j^\top} (\bm{1}-\bm{A w}))^2_- +\lambda\sum_{i=1}^n z_i,\label{eq:SVMCPReformulation3}
    \end{align}
    which, by weak duality, is a relaxation of \eqref{eq:SVMReform1}.
    A further relaxation of \eqref{eq:SVMCPReformulation3} is obtained by simply setting $\bm{d}=\bm{1}$.
This final relaxation is at least as strong as the formulation obtained from  \eqref{eq:substitution} if
    \begin{align*}
        \left\langle \bm{I}_{p\times p}-\sum_{j=1}^q  \bm{A^\top f_j f_j^\top A}, \bm{W}\right\rangle \geq \bm{w^\top}\left(\bm{I}_{p\times p}-\sum_{j=1}^q  \bm{A^\top f_j f_j^\top A}\right)\bm{w},
    \end{align*}
which holds since $\bm{W}\succeq \bm{ww^\top}$ and $\bm{I}_{p\times p}\succeq \bm{A^\top F^\top FA}$, concluding the proof.
\end{proof}

Proposition~\ref{prop:decomposition} ascertains that relaxation \eqref{eq:SVMReform1} is indeed stronger than any direct application of the rank-one convexifications discussed in \S\ref{sec:hull}, \emph{provided that the collection $\mathcal{L}$ is rich enough}. However, this strength may come at a prohibitive cost, as the extreme case where $\mathcal{L}=2^{[n]}$ requires exponentially many constraints, involving cones of dimension up to $p+1$. In practice, we suggest using $\mathcal{L}=\left\{L\subseteq [n]: |L|\leq \kappa\right\}$ for some small enough $\kappa\in \mathbb{Z}_+$: the number $\mathcal{O}(n^\kappa)$ of constraints added is polynomial in the number of datapoints $n$, and the dimension of the cones used is bounded by $\kappa+1$ -- except for one $(p+1)$-dimensional cone used to express \eqref{eq:SVMReform1_W}. While this compromise comes at the expense of relaxation quality, as we show in the next section even setting $\kappa=1$ may result in substantial gains compared with alternatives proposed in the statistics and machine learning literature.

\subsection{Interpretation as regularization}

Consider the case where set $\mathcal{L}$ is taken to be the collection of all singletons. In this case, from Proposition~\ref{prop:decomposition}, we find that relaxation \eqref{eq:SVMReform1} can be interpreted as the relaxation where the objective is ``optimally" expressed as 
\begin{align}\label{eq:1dDecomp}
\|\bm{w}\|_2^2-\sum_{i=1}^n d_i(1-y_i\bm{a_i^\top w})^2+\sum_{i=1}^nd_i\left( \frac{(1-y_i\bm{a_i^\top w})_+^2}{z_i}+(1-y_i\bm{a_i^\top w})_-^2\right)+\lambda\sum_{i=1}^n z_i,
\end{align}
where ``optimally" indicates that vector $\bm{d}$ is chosen to ensure the strongest relaxation while guaranteeing that $\bm{I}-\bm{A^\top\text{Diag}(\bm{d})A}\in \mathcal{S}_+^p$. We now retrieve a representation in the original space of variables of \eqref{eq:SVM01}, i.e., without involving indicator variables $\bm{z}$. To do so, we define a class of one-dimensional functions $\phi$ parametrized by two parameters $d,\lambda\in \R_+$ given by 
\begin{align} \phi(x;d,\lambda):=\min_{0\leq z\leq 1}\,\lambda z-dx^2
	+d\frac{(x)_+^2}{z_i}+d(x)_-^2.\label{eq:maxminPenaltySep}\end{align}
Observe then that the proposed relaxation of SVM problems \eqref{eq:SVM01}, with objective value given by \eqref{eq:1dDecomp}, can be written as 
\begin{align}
    \min_{\bm{w}\in \R^p}\underbrace{\|\bm{w}\|_2^2}_{\text{margin}}+\sum_{i=1}^n \underbrace{\phi\left(1-y_i\bm{a_i^\top w};d_i,\lambda\right)}_{\text{misclassification loss for $i$}},
\end{align}
that is, function $\phi$ controls the penalty for misclassifying datapoint $i$ depending on the violation $1-y_i\bm{a_i^\top w}$. We now provide an explicit description of function $\phi$.
\begin{proposition}\label{prop:loss}
	Assuming $d,\lambda>0$, the optimal value of problem \eqref{eq:maxminPenaltySep} is the non-convex loss function 
	\begin{align}
		\phi(x;d,\lambda) = \begin{cases} 0 &\text{if}\ x\leq 0 \\ 2(\sqrt{\lambda d})x -dx^2 &\text{if}\ \ 0<x\leq \sqrt{\lambda/d}\\ \lambda &\text{if}\  x>\sqrt{\lambda/d}.  \end{cases}
	\end{align}
\end{proposition}
\begin{proof}
	To determine the optimal value of $z$ in \eqref{eq:maxminPenaltySep}, we consider two cases depending on the value of $x$.
	
	\noindent $\bullet$ \emph{Case 1: $x\leq 0$.} In this case, the point is classified on the correct side of the margin. We find that the optimal value is $z^*=0$ as the objective is minimized to zero.  
	
	\noindent$\bullet$ \emph{Case 2: $x\geq 0$.} In this case,setting the derivative of \eqref{eq:maxminPenaltySep} with respect to $z$ to zero, we find that the optimal solution is 
	$$z^*=\min\left\{\sqrt{\frac{d}{\lambda}}x,1\right\},$$
	with objective value $2(\sqrt{\lambda d})x-dx^2$ if $0<x\leq \sqrt{\lambda/d}$, and $\lambda$ if $x>\sqrt{\lambda/d}$.
\end{proof}

Figure~\ref{fig:lossFunctions} depicts the loss function $\phi$ along with alternatives from the literature. The parameter $d$ controls the tradeoff between non-convexity and approximation to the 0-1 loss. In problems with uncertainty and outliers, the concave downward shape of the loss function redistributes the influence from outliers to points near the boundary. Note that for $d=1$, the derived loss is a better approximation of the misclassification loss than the normalized sigmoid, $\psi$-learning, and ramp losses.  This reduces the sensitivity to noise when compared to hinge and other non-convex losses. More importantly, while alternatives \cite{mason1999boosting,shen2003,wu2007robust} require tackling non-convex problems (often settling for local minimizer instead of global), the proposed formulation in Proposition~\ref{prop:formulation} retains convexity and can be solved to global optimality. 

\begin{figure*}[!h]
	\centering
\subfloat[ Normalized sigmoid ]{\includegraphics[width=0.33\textwidth,trim={10cm 5cm 11cm 5cm},clip]{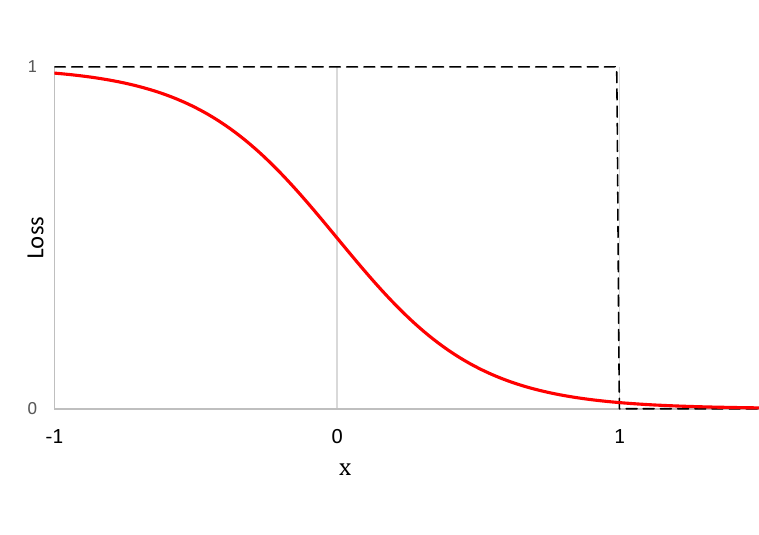}}\hfill
	\subfloat[$\psi$-learning]{\includegraphics[width=0.33\textwidth,trim={10cm 5cm 11cm 5cm},clip]{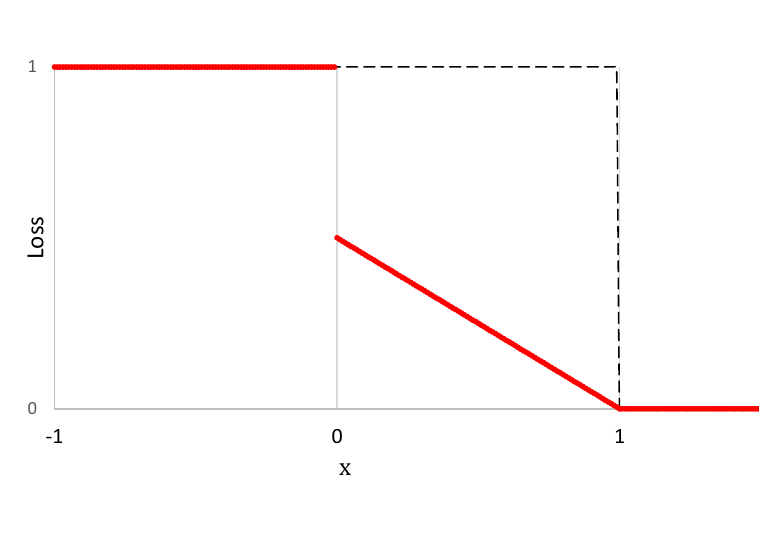}}\hfill
	\subfloat[ Ramp ]{\includegraphics[width=0.33\textwidth,trim={10cm 5cm 11cm 5cm},clip]{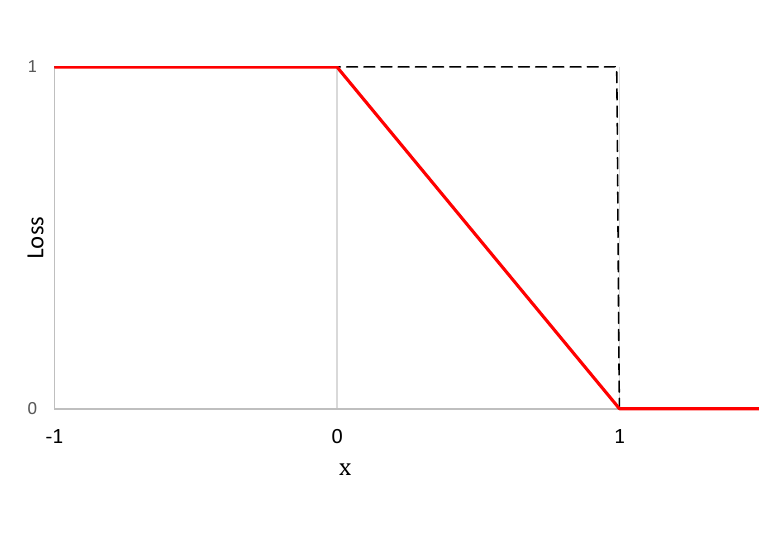}}
	\newline
	\subfloat[ $\phi(\cdot;0.2,1)$ ]{\includegraphics[width=0.33\textwidth,trim={10cm 5cm 11cm 5cm},clip]{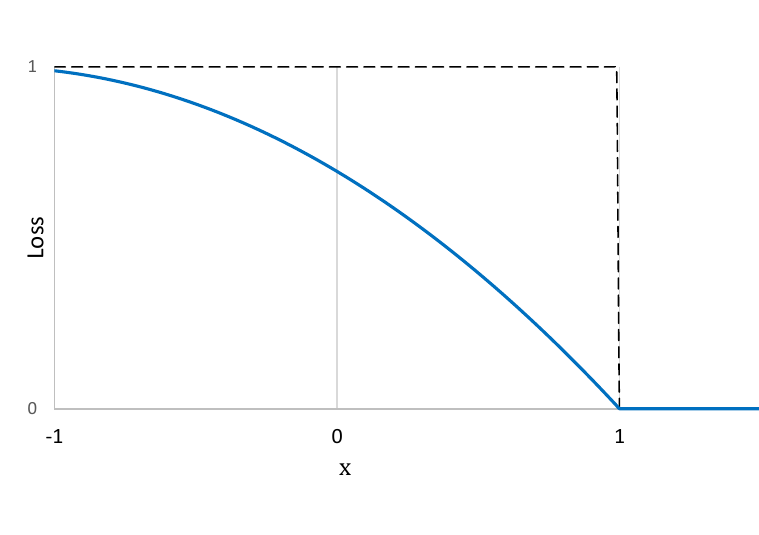}}\hfill
	\subfloat[  $\phi(\cdot;0.5,1)$ ]{\includegraphics[width=0.33\textwidth,trim={10cm 5cm 11cm 5cm},clip]{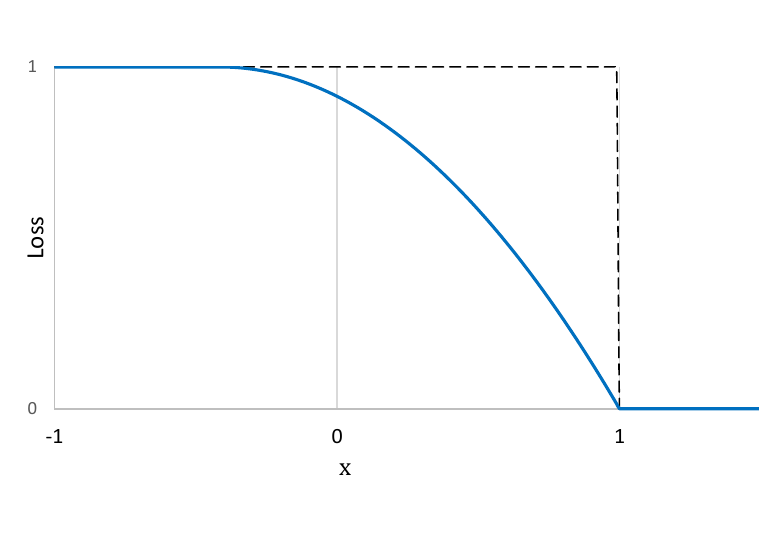}}\hfill
	\subfloat[  $\phi(\cdot;1.0,1)$ ]{\includegraphics[width=0.33\textwidth,trim={10cm 5cm 11cm 5cm},clip]{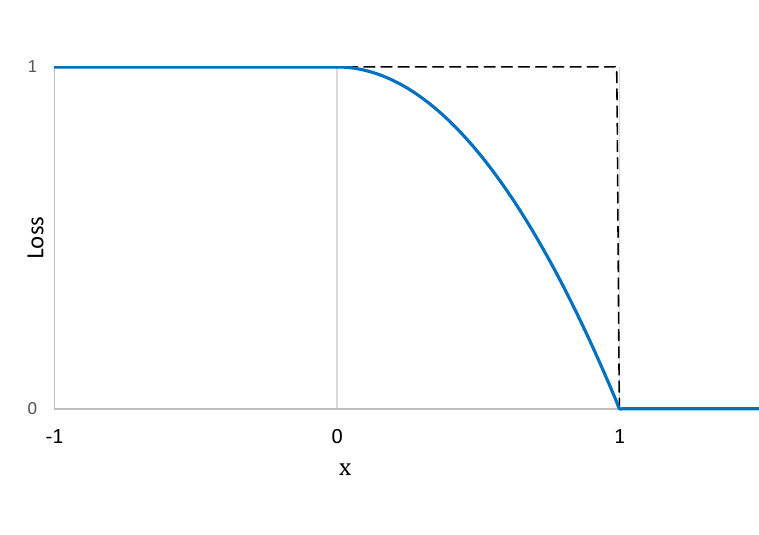}}\hfill
	\caption{\small Non-convex robust losses approximating the 0-1 loss, as a function of $x=y_i\bm{a_i^\top w}$. Top row: loss functions from the literature: the normalized sigmoid loss \citep{mason1999boosting}, the $\psi$-learning loss \citep{shen2003psi} and the ramp loss \citep{wu2007robust}. Bottom row: the derived $\phi$ loss in Proposition~\ref{prop:loss} for different values of hyperparameters $d$ (with $\lambda=1$) . By solving the conic optimization problem \eqref{eq:SVMReform1}, $\bm{d}$ is chosen automatically to ensure the convexity of the ensuing learning problem.}
	\label{fig:lossFunctions}
\end{figure*}


\section{Computational experiments}\label{sec:computations}

We now discuss computations with the proposed formulations in SVM problems. The results are presented in three main subsections. First, \S\ref{sec:compLB} is concerned with the ability to compute lower bounds for SVM problems with 0-1 loss, where the main metrics are bound quality and solution times. The main benchmarks in this part are mixed-integer optimization methods based on big-M formulations as proposed in \cite{brooks2011support}, see \S\ref{sec:bigM}. We point out that since big-M methods fail to produce meaningful bounds in most cases, the methods are tested in problems with $n=100$ datapoints (which would be considered small from a practical standpoint in the context of SVMs). Then, \S\ref{sec:statisticalsynt} is concerned with the statistical performance resulting from the formulations, where the solution of the relaxation is directly used as an estimator. In this case, the main benchmarks are alternatives based on convex optimization, and the results are presented in synthetic instances with up to $n=1,000$ datapoints. Our choice of synthetic instances for these experiments facilitates the comparison of the various estimators, since we have access to the ``ground truth" of the data generation process. Finally, in \S\ref{sec:statisticalReal} we test the methods on datasets from the UCI Machine Learning repository \cite{Dua:2019}. 

Throughout the experiments, we use Gurobi 12.0.0 to solve quadratic and mixed integer optimization problems and Mosek 11.0 to solve semidefinite optimization problems. In all cases default settings of the solvers are used -- the branch-and-bound algorithm with Gurobi uses 20 threads by default. All computations are run on a Laptop with a 12th Gen Intel Core i7-1280P CPU and 32GB RAM. 

\paragraph{Instance generation}\label{sec:instance}
We consider synthetic instances inspired by the instances in \cite{brooks2011support}: the data from each class is generated from a different Gaussian distribution, and the presence of different types of outliers may obfuscate the data. We now describe in detail the data generation process.

Specifically, let $n,p\in \mathbb{Z}_+$ be the dimension parameters, and $\sigma\in \R_+$ represents the standard deviation of the Gaussian distribution. We consider three outlier classes (none, clustered, spread). We first generate a direction $\bm{d}\in \R^p$ where each entry is generated independently from a uniform distribution in $[-1,1]$. We then set two centroids $\bm{\alpha_1}=0.5\bm{d}/\|\bm{d}\|_2$ and $\bm{\alpha_{-1}}=-0.5\bm{d}/\|\bm{d}\|_2$, and note that they are always one unit apart. Then we generate $n$ points $\bm{a_i}\in \R^{p+1}$ where $(a_i)_1=1$ (to account for an intercept) and the remaining coordinates depend on the outlier class, and are generated as follows:

\noindent$\bullet$ \textbf{none}  With $0.5$ probability the remainder coordinates are generated from $\mathcal{N}(\bm{\alpha_1},\sigma^2\bm{I})$ and set the label $y_i=1$, and with $0.5$ probability the remainder coordinates are generated from $\mathcal{N}(\bm{\alpha_{-1}},\sigma^2\bm{I})$ and set the label $y_i=-1$. The top row of Figure~\ref{fig:dataset} illustrates a generation of this class of dataset with $p=2$ and $\sigma\in \{0.2,0.5,1.0\}$.

\noindent$\bullet$ \textbf{clustered} With $0.45$ probability the remainder coordinates are generated from $\mathcal{N}(\bm{\alpha_1},\sigma^2\bm{I})$ and set the label $y_i=1$, with $0.45$ probability the remainder coordinates are generated from $\mathcal{N}(\bm{\alpha_{-1}},\sigma^2\bm{I})$ and set the label $y_i=-1$, and with $0.1$ probability the remainder coordinates are generated from $\mathcal{N}(10\bm{\alpha_{-1}},0.001\sigma^2\bm{I})$ and set the label $y_i=1$. Figure~\ref{fig:dataset} (d) illustrates a generation of this class of dataset with $p=2$ and $\sigma=0.5$.

\noindent$\bullet$ \textbf{spread} An expected 45\% of the points are generated from $\mathcal{N}(\bm{\alpha_1},\sigma^2\bm{I})$ and set the label $y_i=1$, 45\% are generated from $\mathcal{N}(\bm{\alpha_{-1}},\sigma^2\bm{I})$ and set the label $y_i=-1$, 5\% of the points are generated from $\mathcal{N}(\bm{\alpha_1},100\sigma^2\bm{I})$ and set the label $y_i=1$ and 5\% of the points are generated from $\mathcal{N}(\bm{\alpha_{-1}},100\sigma^2\bm{I})$ and set the label $y_i=-1$.  Figure~\ref{fig:dataset} (e) illustrates a generation of this class of dataset with $p=2$ and $\sigma=0.5$. 

In all cases, the \emph{ideal} Bayes classifier is the line perpendicular to $\bm{\bm{\alpha_1}}-\bm{\alpha_{-1}}$ going through the origin, that is, $\bm{\hat w^\top}=(0\; \bm{d^\top})$. 

\begin{figure*}[!h]
	\centering
\subfloat[ \textbf{none} with $\sigma=0.2$]{\includegraphics[width=0.33\textwidth,trim={10cm 5.5cm 10cm 5.5cm},clip]{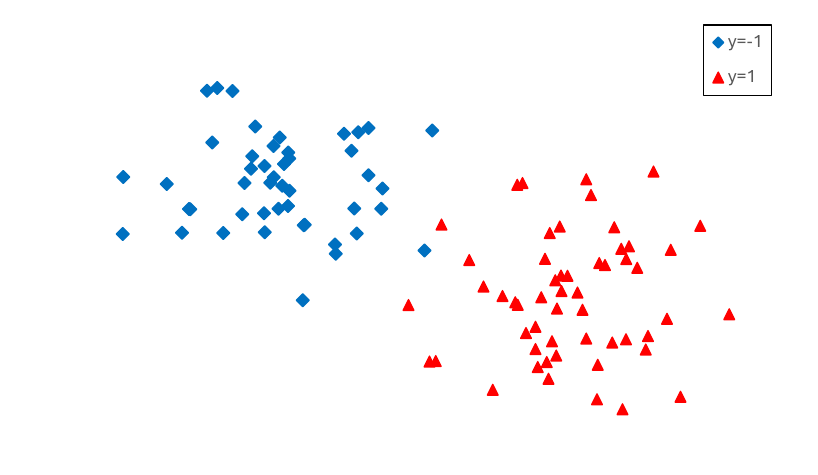}}\hfill
	\subfloat[ \textbf{none} with $\sigma=0.5$]{\includegraphics[width=0.33\textwidth,trim={10cm 5.5cm 10cm 5.5cm},clip]{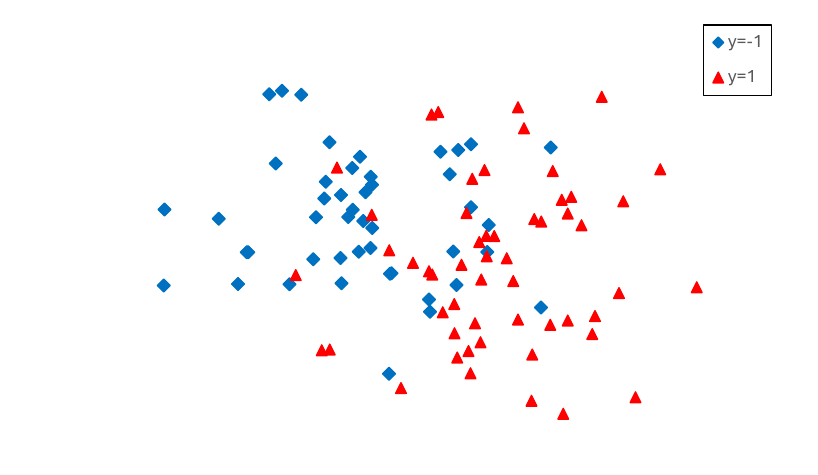}}\hfill
	\subfloat[ \textbf{none} with $\sigma=1.0$]{\includegraphics[width=0.33\textwidth,trim={10cm 5.5cm 10cm 5.5cm},clip]{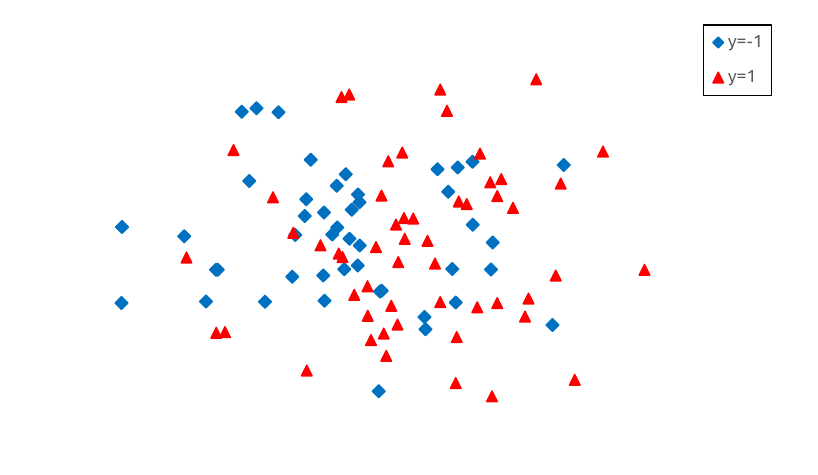}}
	\newline
	\hfill\subfloat[ \textbf{clustered} with $\sigma=0.5$]{\includegraphics[width=0.33\textwidth,trim={10cm 5.5cm 10cm 5.5cm},clip]{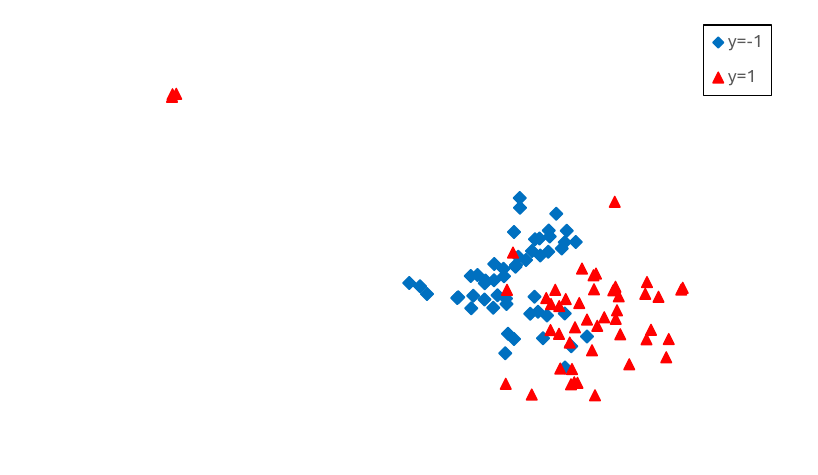}}\hspace{2cm}
	\subfloat[ \textbf{spread} with $\sigma=0.5$]{\includegraphics[width=0.33\textwidth,trim={10cm 5.5cm 10cm 5.5cm},clip]{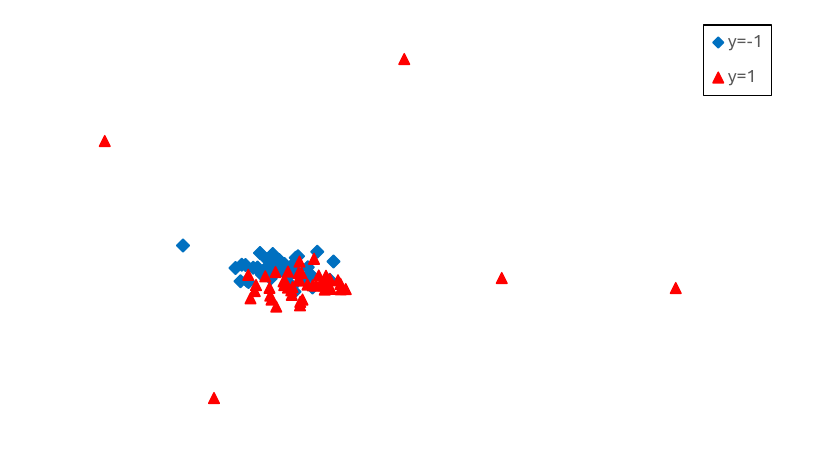}}\hfill
	\caption{\small Sample synthetic datasets with $p=2$ and $n=100$.}
	\label{fig:dataset}
\end{figure*}

\subsection{Computing lower bounds}\label{sec:compLB}
In this section we test the ability of the relaxations proposed in Proposition~\ref{prop:formulation} to generate lower bounds for the cardinality constrained problem \eqref{eq:BigMCardinality}, and compare them with standard mixed integer optimization software. In this section we provide only a limited discussion of the runtimes of the conic relaxations and elaborate on this topic in later subsections.

\paragraph{Methods} We compare three methods to produce lower bounds for \eqref{eq:BigMCardinality}.

\noindent $\bullet$ \textbf{gurobi} Default Gurobi based on formulation \eqref{eq:BigMCardinality}, with $M=1,000$ and a time limit of 10 minutes.

\noindent $\bullet$ \textbf{conic1} The semidefinite formulation given in Proposition~\ref{prop:formulation} where $\mathcal{L}=\{\{i\}:  i\in [n]\}$ is the collection of all singletons.

\noindent $\bullet$ \textbf{conic2} The semidefinite formulation given in Proposition~\ref{prop:formulation} where $\mathcal{L}=\{\{i,j\}:  1\leq i<j\leq n\}$ is the collection of all sets of cardinality two.

\paragraph{Instances} We report results on synthetic instances for all three classes of outliers, where the number of datapoints is $n=100$, number of features $p\in\{10,20,30,50\}$, standard deviations $\sigma\in\{0.5, 1.0\}$ and the maximum number of misclassifications $k\in \{10, 20, 30\}$. For each combination of parameters, we generate five instances.

In addition of the settings reported in depth in the paper, we briefly comment on other settings we also tested in preliminary computations: $\bullet$ If $p\leq 5$, then the MIO problems are in most cases infeasible (the conic relaxations are still feasible); $\bullet$ If $\sigma=0.2$, then the data is almost perfectly separable and the MIO method finds optimal solutions in seconds (and is thus preferable to the continuous relaxations proposed); $\bullet$ If $n\geq 200$, then the MIO method does not terminate in any instance and results in large optimality gaps, regardless of the value of $\sigma$ (the conic relaxations result in better lower bounds, similarly to the setting $p=50$ reported here).

\paragraph{Summary of results} While MIO methods can sometimes quickly solve problems to optimality if the number of features $p$ or the noise $\sigma$ are small, they fail in most cases with large optimality gaps. Conic relaxations are, in fact, capable of producing better lower bounds in cases with large values of $p$, while requiring a fraction of the computational time and resources. In particular, \textbf{conic1} can consistently achieve reasonable optimality gaps in less than a second. 

\paragraph{In depth discussion} Tables~\ref{tab:computational05} and \ref{tab:computational10} present the results in instances with $\sigma=0.5$ and $\sigma=1.0$, respectively. Each row represents an average over at most 15 instances with similar parameters, with 5 instances corresponding to each of the three classes of datasets. Note that we only report results for instances where \textbf{gurobi} is able to find a feasible solution, therefore some rows may include fewer than 15 instances. In each row we indicate: the number of features $p$ and budget $k$ for misclassifications; for \textbf{gurobi} we indicate the average time in seconds (with timeouts counting as 600 seconds), the average gap at termination, the average number of branch-and-bound nodes explored, the number of instances where a feasible solution is found and the number of instances where an optimal solution is found; for the \textbf{conic} relaxations, we indicate the time required to solve the relaxation and the optimality gap computed as $\texttt{gap}:=\left(\zeta_{gurobi}-\zeta_{conic}\right)/\zeta_{gurobi}$, where $\zeta_{gurobi}$ is the objective value of the best solution found by \textbf{gurobi} and $\zeta_{conic}$ is the objective value of the relaxation.

\begin{table}[!h]
\begin{center}
	\caption{\small Computational results with $\sigma=0.5$ in synthetic feasible instances. Here, the ``\# feas" column shows the number of instances (out of 15) where a feasible solution was found, and the averages reported correspond to these instances only. We mark in \textbf{bold} entries where a conic relaxation can produce on average better bounds than branch-and-bound methods. The symbol ``$\dagger$" indicates that numerical issues were encountered.} 
	\label{tab:computational05}
	\begin{tabular}{ c c |c  c c c c | c c| c c}
		\hline
		\multirow{2}{*}{$\bm{p}$}&\multirow{2}{*}{$\bm{k}$}&\multicolumn{5}{c|}{\underline{\textbf{gurobi}}}&\multicolumn{2}{c|}{\underline{\textbf{conic1}}}&\multicolumn{2}{c}{\underline{\textbf{conic2}}}\\
        &&\textbf{time}&\textbf{gap}&\textbf{nodes}&\textbf{\# feas}&\textbf{\# opt}&\textbf{time}&\textbf{gap}&\textbf{time}&\textbf{gap}\\
		\hline
		\multirow{3}{*}{10}&10&51.2&0.0\%&650,246&13&13&0.0&95.6\%&3.3&81.5\%\\	&20&147.9&0.9\%&2,358,861&15&14&0.0&84.6\%&3.2&42.1\%\\        &30&211.6&0.0\%&3,662,694&15&15&0.0&74.3\%&2.8&28.2\%\\
		\hline	
        \multirow{3}{*}{20}&10&218.0&1.2\%&2,366,233&15&14&0.1&84.5\%&12.6&54.4\%\\	&20&588.1&43.7\%&7,300,333&15&1&0.1&70.6\%&\textbf{11.5}&\textbf{25.7\%}\\        &30&600.0&48.3\%&7,875,416&15&0&0.1&62.9\%&\textbf{11.0}&\textbf{12.2\%}\\
		\hline	
        \multirow{3}{*}{30}&10&534.6&33.3\%&4,281,846&15&4&0.2&73.9\%&23.1&33.4\%\\	&20&600.0&69.7\%&5,419,954&15&0&\textbf{0.2}&\textbf{62.3\%}&\textbf{23.4}&\textbf{13.7\%}\\        &30&600.0&66.0\%&5,955,912&15&0&\textbf{0.2}&\textbf{55.6\%}&\textbf{21.2}&\textbf{7.9\%}\\
		\hline	
        \multirow{3}{*}{50}&10&600.0&53.9\%&3,689,297&15&0&\textbf{0.6}&\textbf{46.6\%}&$\dagger$&$\dagger$\\	&20&600.0&74.8\%&4,215,349&15&0&\textbf{0.6}&\textbf{41.6\%}&$\dagger$&$\dagger$\\        &30&600.0&74.9\%&4,278,767&15&0&\textbf{0.6}&\textbf{37.7\%}&$\dagger$&$\dagger$\\
		\hline	
	\end{tabular}
    \end{center}
\end{table}

\begin{table}[!h]
\begin{center}
	\caption{\small Computational results with $\sigma=1.0$ in synthetic feasible instances.  Here, the ``\# feas" column shows the number of instances (out of 15) where a feasible solution was found, and the averages reported correspond to these instances only. We mark in \textbf{bold} entries where a conic relaxation can produce on average better bounds than branch-and-bound methods. The symbol ``$\dagger$" indicates that numerical issues were encountered.} 
	\label{tab:computational10}
	\begin{tabular}{ c c |c  c c c c | c c| c c}
		\hline
		\multirow{2}{*}{$\bm{p}$}&\multirow{2}{*}{$\bm{k}$}&\multicolumn{5}{c|}{\underline{\textbf{gurobi}}}&\multicolumn{2}{c|}{\underline{\textbf{conic1}}}&\multicolumn{2}{c}{\underline{\textbf{conic2}}}\\
        &&\textbf{time}&\textbf{gap}&\textbf{nodes}&\textbf{\# feas}&\textbf{\# opt}&\textbf{time}&\textbf{gap}&\textbf{time}&\textbf{gap}\\
		\hline
		\multirow{3}{*}{10}&10&162.7&0.0\%&1,538,726&1&1&0.0&99.6\%&2.9&99.1\%\\	&20&575.9&83.2\%&7,395,964&14&1&0.0&94.9\%&3.2&89.8\%\\        &30&600.0&46.5\%&10,689,069&15&1&0.0&77.4\%&3.1&59.4\%\\
		\hline	
        \multirow{3}{*}{20}&10&600.0&95.1\%&4,016,287&10&0&0.1&97.1\%&\textbf{13.0}&\textbf{95.1\%}\\	&20&600.0&85.0\%&6,543,946&15&0&\textbf{0.1}&\textbf{76.8\%}&\textbf{11.4}&\textbf{62.8\%}\\        &30&600.0&76.7\%&7,634,594&15&0&\textbf{0.1}&\textbf{58.2\%}&\textbf{11.4}&\textbf{36.7\%}\\
		\hline	
        \multirow{3}{*}{30}&10&600.0&80.6\%&3,855,544&15&0&\textbf{0.2}&\textbf{79.3\%}&\textbf{22.4}&\textbf{68.3\%}\\	&20&600.0&85.3\%&4,728,503&15&0&\textbf{0.2}&\textbf{58.0\%}&\textbf{22.0}&\textbf{37.5\%}\\        &30&600.0&79.9\%&5,484,409&15&0&\textbf{0.2}&\textbf{40.0\%}&\textbf{21.2}&\textbf{14.5\%}\\
		\hline	
        \multirow{3}{*}{50}&10&600.0&78.2\%&3,196,022&15&0&\textbf{0.7}&\textbf{44.2\%}&\textbf{67.2}&\textbf{25.5\%}\\	&20&600.0&85.1\%&3,902,330&15&0&\textbf{0.6}&\textbf{31.6\%}&$\dagger$&$\dagger$\\        &30&600.0&83.7\%&4,305,493&15&0&\textbf{0.7}&\textbf{20.6\%}&$\dagger$&$\dagger$\\
		\hline	
	\end{tabular}
    \end{center}
\end{table}

We observe that \textbf{gurobi} is able to solve some problems optimally through extensive enumeration, particularly in settings where the number of features is small ($p=10$), there is small to medium noise ($\sigma=0.5$) and the maximum number of misclassifications is limited ($k=10$). However, in most other scenarios the resulting end gaps are consistently larger than 50\%, and in several cases larger than 80\% -- these results are similar to those reported in \cite{brooks2011support}.
Interestingly, conic relaxations seem to excel in the settings where \textbf{gurobi} performs worse, and struggle when \textbf{gurobi} performs best.  Indeed, if $p=10$, then Gurobi can solve most of the problems to optimality, while conic relaxations result in large optimality gaps; in some cases as large as 99\%. On the other hand, if $p\geq 30$, then \textbf{gurobi} struggles to prove optimality gaps below 50\% and conic relaxations consistently produce better bounds: for example, if $p=30$ and $\sigma=1.0$, then the average optimality gap proven by \textbf{gurobi} of 79.9\% is substantially reduced to 40.0\% and 14.5\% by \textbf{conic1} and \textbf{conic2}, respectively. The better bounds of the conic relaxations are also achieved in a fraction of the computational time: while \textbf{gurobi} uses all of 10 minutes and 20 threads, \textbf{conic1} and \textbf{conic2} terminate within one and 30 seconds, respectively, and run in one thread. 

\paragraph{Numerical issues} We point out that in instances with $p=50$, \textbf{conic2} results in numerical issues in several instances. In most situations with numerical issues, the Mosek solver terminates with an ``unknown" status and fails to produce a solution; in some cases, it also produces a solution that, upon inspection, is extremely low quality. We observe that different standardizations of the data may resolve numerical difficulties in some of the instances, while creating issues in other instances; in the table, we report results without standardization. Recall that \textbf{conic2} requires the addition of ${n\choose 2}=4,950$ positive semi-definite constraints of order 3, each requiring the addition of two additional variables and fairly dense linear constraints involving all $p=50$ of the original variables $\bm{w}$. In contrast, \textbf{conic1}, which requires only $n=100$ rotated cone constraints, does not suffer from numerical instabilities in these instances (or any of the larger instances considered in later subsections).

\subsection{Statistical performance in synthetic instances}\label{sec:statisticalsynt}
In this section we test the ability of the relaxations proposed in Proposition~\ref{prop:formulation} to produce good estimators to robust SVM problems. In each experiment, we generate a training and validation set (both contaminated by outliers for classes ``clustered" and ``spread"), each with $n$ observations, and a testing set without outliers and $100,000$ observations. For each method tested, we solve the training problem on the training set for 100 different values of the hyperparameters, choose the solution that results in fewer misclassified points in validation, and report the (out-of-sample) misclassification rate on the testing set, as well as the total time required to perform cross-validation (i.e., solving 100 training problems). For each combination of instance class and dimensions, we repeat this process 20 times and report averages and standard deviations across all replications.

\paragraph{Methods} We compare the following approaches for learning SVMs.

\noindent $\bullet$ \textbf{hinge} Using the hinge loss SVM \eqref{eq:hinge}, formulated as a convex quadratic program and solved using Gurobi. For cross-validation, we set $\lambda=\beta/(1-\beta)$ where the 100 values of $\beta$ are uniformly selected in the interval $(0,1)$.

\noindent $\bullet$ \textbf{robustSVM} The method for robust SVMs proposed in \cite{bertsimas2019robust}, which calls for the solution of 
	\begin{align*}
		\min_{\bm{w}\in \R^p,\bm{\xi}\in \R_+^n}\;&\sum_{i=1}^n \xi_i\\
		\text{s.t.}\:&y_i\left(\bm{a_i^\top w}\right)-\lambda \|\bm{w}\|_1\geq 1-\xi_i\quad \forall i\in \{1,\dots,n\},
	\end{align*}
solved as a linear program (after the usual linearization of the absolute values) using Gurobi.
Note that the choice of the $\ell_1$ norm to penalize $\bm{w}$ matches the one used in the experiments in \cite{bertsimas2019robust}. The parameters $\lambda$ for cross-validation are identical to the \textbf{hinge} method. 

\noindent $\bullet$ \textbf{0-1 loss} The cardinality constrained problem \eqref{eq:BigMCardinality} solved using Gurobi with a time limit of three minutes per cardinality choice $k$. Due to the large computational burden and the inability to produce meaningful optimality gaps in instances with $n\geq 200$, we use this method on instances with $n=100$ only. For cross-validation, we train for all values\footnote{While only 50 values of the hyperparameter are tested in instances with $n=100$ (instead of 100 like other methods), no good solutions can be obtained for the other choices; this set of parameters is exhaustive.} $k\in [n/2]$.  

\noindent $\bullet$ \textbf{conic1} The semidefinite relaxation given in Proposition~\ref{prop:formulation} for a cardinality constrained problem (solved using Mosek), where $\mathcal{L}=\{\{i\}:  i \in [n]\}$ is the collection of all singletons. For cross-validation, we select 100 uniformly-spaced right-hand side values $k$ in the interval $(0,n/2)$, loosely corresponding to the intuition that good estimators should misclassify at most half the points.

\noindent $\bullet$ \textbf{conic2} The semidefinite relaxation given in Proposition~\ref{prop:formulation} for a cardinality constrained problem (solved using Mosek), where $\mathcal{L}=\{\{i,j\}: 1\leq i<j\leq n\}$ is the collection of all sets of cardinality two. Due to the large computational burden, we use this method on instances with $n=100$ only. The choice of hyperparameters is identical to \textbf{conic1}.

\noindent $\bullet$ \textbf{hinge+conic1} A combination of the two methods, where 50 hinge SVMs are solved and 50 conic SVMs are solved (resulting in the same number of 100 total models tested), and the best model is selected via cross-validation. The choice of hyperparameters tested is identical to those for methods \textbf{hinge} and \textbf{conic1}, except that only 50 are tested for each method.

\noindent $\bullet$ \textbf{bayes} The bayes estimator corresponding to selecting $\bm{\hat w^\top}=(0\; \bm{d^\top})$. This ``method" serves as a benchmark as it is the best possible estimator if the data-generation process is known, but is evidently not implementable in practice.

\paragraph{Instances} We report results on synthetic instances for all three classes of outliers, where the number of datapoints is $n\in \{100,200,500,1000\}$, number of features $p\in\{3,5,10,30\}$, standard deviations $\sigma\in\{0.2,0.5, 1.0\}$. For each combination of parameters, we generate 20 instances.

\paragraph{Summary of results} While \textbf{hinge} performs well in instances without outliers, it can completely break down with just a few outliers. In contrast, the proposed \textbf{conic1} is able to achieve a robust performance across all settings within a reasonable computational time. The combination of \textbf{hinge+conic1} results in an ideal estimator in most scenarios.

\paragraph{Discussion on the quality of the estimators}

Tables~\ref{tab:misclassification100Synt} and \ref{tab:misclassification200Synt} show the results in synthetic instances with data points $n=100$ and $n\geq 200$, respectively. The reason for presenting the case $n=100$ separately is two-fold: \textit{(i)} methods \textbf{0-1 loss} and \textbf{conic2} are only used with $n=100$ due to their computational burden; \textit{(ii)} instances with fewer datapoints (and thus less signal-to-noise ratio) exhibit a more pronounced pathological behavior in the presence of outliers.

\begin{table}[!h]
\begin{center}
\caption{Out-of-sample misclassification rate with synthetic instances, as a function of outlier generation and noise parameter $\sigma$, for problems with $n=100$ and $p\in \{3,5,10,30\}$. We highlight in \textbf{bold} the best method for a given class of instances.} 
\label{tab:misclassification100Synt}
\setlength{\tabcolsep}{1pt}
\begin{tabular}{ c c | c |c |c }
\hline
$\bm{\sigma}$&\textbf{method}&\textbf{none}&\textbf{clustered}&\textbf{spread}\\
\hline
\multirow{6}{*}{0.2}&hinge&$\bm{1.0\%\pm0.4\%}$&$9.2\%\pm 17.0\%$&$1.3\%\pm 0.7\%$\\
&robustSVM&$1.9\%\pm0.1\%$&$8.0\%\pm 12.5\%$&$2.3\%\pm 1.5\%$\\
&0-1 loss&$1.1\%\pm0.4\%$&$\bm{1.2\%\pm 0.7\%}$&$\bm{1.2\%\pm 0.6\%}$\\
&conic1&$1.3\%\pm0.8\%$&$2.4\%\pm 2.8\%$&$1.6\%\pm 0.9\%$\\
&conic2&$1.3\%\pm0.7\%$&$1.9\%\pm 1.0\%$&$1.6\%\pm 0.9\%$\\
&hinge+conic1&$1.2\%\pm0.7\%$&$1.8\%\pm 2.0\%$&$1.5\%\pm 0.9\%$\\
&bayes&$0.6\%\pm0.0\%$&$0.6\%\pm 0.0\%$&$0.6\%\pm 0.0\%$\\
\hline
\multirow{6}{*}{0.5}&hinge&$17.9\%\pm1.9\%$&$27.5\%\pm 11.5\%$&$19.9\%\pm 2.7\%$\\
&robustSVM&$19.4\%\pm3.0\%$&$28.3\%\pm 10.7\%$&$21.5\%\pm 4.0\%$\\
&0-1 loss&$18.8\%\pm2.3\%$&$\bm{19.7\%\pm 2.9\%}$&$19.4\%\pm 2.8\%$\\
&conic1&$18.2\%\pm1.7\%$&$21.7\%\pm 6.1\%$&$19.0\%\pm 2.5\%$\\
&conic2&$18.7\%\pm2.3\%$&$20.8\%\pm 4.4\%$&$19.5\%\pm 2.6\%$\\
&hinge+conic1&$\bm{17.8\%\pm1.5\%}$&$20.8\%\pm 5.6\%$&$\bm{18.9\%\pm 2.7\%}$\\
&bayes&$15.9\%\pm0.1\%$&$15.9\%\pm 0.1\%$&$15.9\%\pm 0.1\%$\\
\hline
\multirow{6}{*}{1.0}&hinge&$\bm{34.0\%\pm2.6\%}$&$40.6\%\pm 6.8\%$&$36.4\%\pm 3.8\%$\\
&robustSVM&$35.3\%\pm3.7\%$&$41.8\%\pm 6.5\%$&$37.3\%\pm 4.3\%$\\
&0-1 loss&$35.1\%\pm2.7\%$&$\bm{36.9\%\pm 5.5\%}$&$35.4\%\pm 3.1\%$\\
&conic1&$34.4\%\pm2.7\%$&$37.7\%\pm 4.6\%$&$\bm{35.1\%\pm 3.0\%}$\\
&conic2&$34.6\%\pm2.5\%$&$38.3\%\pm 5.6\%$&$35.2\%\pm 3.0\%$\\
&hinge+conic1&$\bm{34.0\%\pm2.6\%}$&$37.4\%\pm 4.7\%$&$\bm{35.1\%\pm 3.1\%}$\\
&bayes&$30.9\%\pm0.1\%$&$30.9\%\pm 0.1\%$&$30.9\%\pm 0.1\%$\\
\hline

\end{tabular}
\end{center}
\end{table}

\begin{table}[!h]
\begin{center}
\caption{Out-of-sample misclassification rate with synthetic instances, as a function of outlier generation and noise parameter $\sigma$, for problems with $n\in \{200,500,1000\}$ and $p\in \{3,5,10,30\}$. We highlight in \textbf{bold} the best method for a given class of instances.} 
\label{tab:misclassification200Synt}
\setlength{\tabcolsep}{1pt}
\begin{tabular}{ c c | c |c |c }
\hline
$\bm{\sigma}$&\textbf{method}&\textbf{none}&\textbf{clustered}&\textbf{spread}\\
\hline
\multirow{5}{*}{0.2}&hinge&$\bm{0.7\%\pm0.1\%}$&$2.1\%\pm 7.1\%$&$\bm{0.8\%\pm 0.3\%}$\\
&robustSVM&$1.0\%\pm0.4\%$&$2.2\%\pm 4.8\%$&$1.3\%\pm 0.6\%$\\
&conic1&$0.8\%\pm0.3\%$&$1.1\%\pm 1.2\%$&$0.9\%\pm 0.4\%$\\
&hinge+conic1&$0.8\%\pm0.3\%$&$\bm{0.9\%\pm 1.0\%}$&$\bm{0.8\%\pm 0.3\%}$\\
&bayes&$0.6\%\pm0.0\%$&$0.6\%\pm 0.0\%$&$0.6\%\pm 0.0\%$\\
\hline
\multirow{5}{*}{0.5}&hinge&$\bm{16.5\%\pm0.7\%}$&$19.8\%\pm 6.5\%$&$17.6\%\pm 1.7\%$\\
&robustSVM&$16.9\%\pm1.2\%$&$20.6\%\pm 6.4\%$&$18.2\%\pm 2.3\%$\\
&conic1&$16.6\%\pm0.8\%$&$17.8\%\pm 2.9\%$&$\bm{17.0\%\pm 1.2\%}$\\
&hinge+conic1&$\bm{16.5\%\pm0.7\%}$&$\bm{17.3\%\pm 2.4\%}$&$17.1\%\pm 1.2\%$\\
&bayes&$15.9\%\pm0.1\%$&$15.9\%\pm 0.1\%$&$15.9\%\pm 0.1\%$\\
\hline
\multirow{5}{*}{1.0}&hinge&$\bm{31.9\%\pm1.2\%}$&$35.1\%\pm 4.4\%$&$33.7\%\pm 2.4\%$\\
&robustSVM&$32.4\%\pm1.7\%$&$36.0\%\pm 5.0\%$&$34.0\%\pm 2.8\%$\\
&conic1&$32.1\%\pm1.3\%$&$33.8\%\pm 2.9\%$&$\bm{32.4\%\pm 1.4\%}$\\
&hinge+conic1&$32.0\%\pm1.3\%$&$\bm{33.5\%\pm 2.8\%}$&$32.5\%\pm 1.5\%$\\
&bayes&$30.9\%\pm0.1\%$&$30.9\%\pm 0.1\%$&$30.9\%\pm 0.1\%$\\
\hline
\end{tabular}
\end{center}
\end{table}

We observe that in regimes with no outliers, \textbf{hinge} tends to be the better estimator but the 0-1 and conic losses result in a similar performance, with a difference of less than one percentage point in the misclassification rate (and \textbf{robustSVM} is the worst estimator in this setting). However, in regimes with outliers, the 0-1 loss and the conic proxies result in significantly better performance than \textbf{hinge} and \textbf{robustSVM}, having better average out-of-sample misclassification \emph{and} smaller variance. This phenomenon is particularly visible in Table~\ref{tab:misclassification100Synt} in instances with ``clustered" outliers and $\sigma=0.2$, and Figure~\ref{fig:oos} depicts detailed results in this setting. We observe that the large variance of the \textbf{hinge} estimator is due to a subset of instances where the estimator breaks down, resulting in solutions with over 50\% misclassification rate. This phenomenon is particularly acute in low signal-to-noise settings caused by a small number of datapoints. In contrast, \textbf{conic1} does not exhibit such a pathological behavior and performs well consistently. 

\begin{figure}[!h]
	\centering
	\subfloat[${n=100}$]{\includegraphics[width=0.40\columnwidth,trim={13cm 6cm 14cm 6cm},clip]{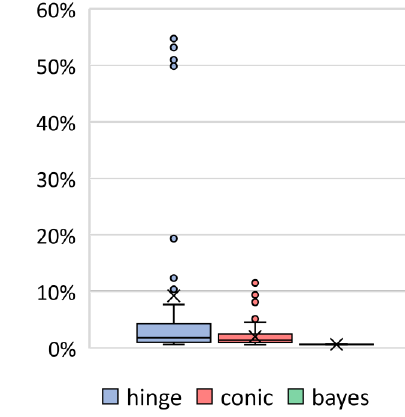}}\hfill
	\subfloat[${n=200}$]{\includegraphics[width=0.40\columnwidth,trim={13cm 6cm 14cm 6cm},clip]{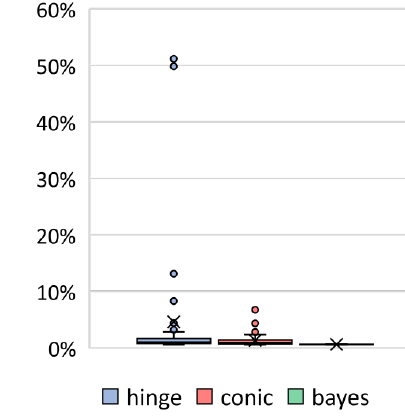}}\hfill
	\newline
	\subfloat[${n=500}$]{\includegraphics[width=0.40\columnwidth,trim={13cm 6cm 14cm 6cm},clip]{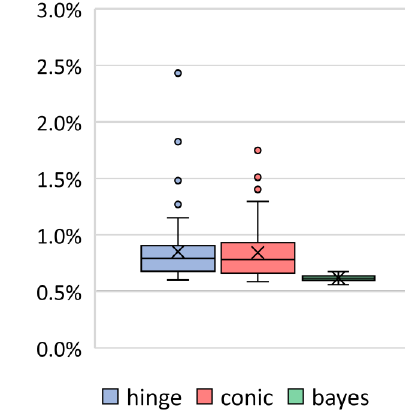}}\hfill
	\subfloat[${n=1,000}$]{\includegraphics[width=0.40\columnwidth,trim={13cm 6cm 14cm 6cm},clip]{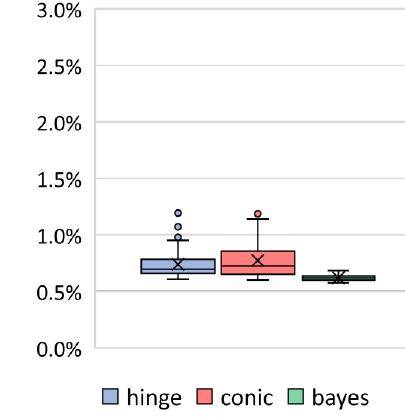}}
	\caption{Distribution of out-of-sample misclassification for data with clustered outliers and $\sigma=0.2$, as a function of the number of datapoints $n$. In instances with small $n$ (top row), the hinge estimator has a probability of breaking down, resulting in out-of-sample misclassifications above 50\%; \emph{the conic loss reduces the average misclassification rate by an order-of-magnitude}. Moreover, the conic estimator performs consistently well in all settings.}
	\label{fig:oos}
\end{figure}

In instances with $n=100$ and clustered outliers, the \textbf{0-1 loss} is the best estimator, but \textbf{conic1} and \textbf{conic2} have a competitive performance as well -- in the case of \textbf{conic1}, at a fraction of the computational time (see next section). In settings with no outliers or spread outliers, \textbf{conic1} is competitive and can even outperform the \textbf{0-1 loss}. Moreover, the results also indicate that training \textbf{hinge+conic1} is a simple approach that results in the best of the worlds without any obvious drawback. Therefore, we do not advocate using \textbf{conic1} as a replacement of standard SVMs with the hinge loss, but rather as a complement in high-stakes situations where outliers could be present and potentially lead to harmful decisions.

\paragraph{Discussion on computational times}

Table~\ref{tab:solTimesSynt} presents the computational times required to select a model using cross-validation as described above. We found that the computational times depend on $n$ and $p$ but not on the instance class, thus we present aggregated results across all instance classes. Note that there is an obvious discrepancy between the computational times reported before in Tables~\ref{tab:computational05} and \ref{tab:computational10}: indeed, rather than describing the time to solve a single SVM instance, Table~\ref{tab:solTimesSynt} reports the time of carrying out a complete cross-validation pass (involving solving 100 instances and additional tasks). Thus, a rather unassuming 20 seconds to solve a single instance turns into an hour-long computation to find the best hyperparameter via cross-validation.  

\begin{table}[!h]
\begin{center}
\caption{Average solution times and standard deviations (in seconds) across all instance classes required to solve the 100 training problems for cross-validation. } 
\label{tab:solTimesSynt}
\begin{tabular}{c c | c |c |c |c}
\hline
$\bm{m}$&\textbf{method}&$\bm{p=3}$&$\bm{p=5}$&$\bm{p=10}$&$\bm{p=30}$\\
\hline
\multirow{5}{*}{$\bm{100}$}&hinge&$0.7\pm 0.2$ &$0.7\pm 0.2$&$0.7\pm 0.1$&$0.9\pm0.2$\\
&robustSVM&$0.4\pm 0.1$ &$0.4\pm 0.1$&$0.4\pm 0.1$&$0.7\pm0.2$\\
&0-1 loss&$560.5\pm 742.4$ &$910.2\pm 1,049.2$&$2,067.7\pm 1,539.4$&$3,582.6\pm661.2$\\
&conic1&$2.4\pm 0.2$&$2.7\pm 0.3$&$4.0\pm 0.7$&$39.3\pm 10.4$\\
&conic2&$354.8\pm 198.8$&$377.9\pm 212.8$&$523.1\pm 368.3$&$3,210.9\pm 708.1$\\
&hinge+conic1&$4.1\pm 0.1$&$4.3\pm 0.1$&$5.3\pm 0.4$&$14.6\pm 1.6$\\
\hline
\multirow{4}{*}{$\bm{200}$}&hinge&$1.1\pm 0.3$ &$1.4\pm 0.5$&$3.2\pm 1.4$&$3.4\pm 1.5$\\
&robustSVM&$0.4\pm 0.1$ &$0.6\pm 0.2$&$0.6\pm 0.2$&$1.2\pm 0.3$\\
&conic1&$3.3\pm 0.7$ &$4.7\pm 2.3$&$7.2\pm 0.8$&$55.0\pm 8.1$\\
&hinge+conic1&$12.4\pm 1.7$ &$19.6\pm 6.6$&$11.7\pm 2.0$&$42.6\pm 4.4$\\
\hline
\multirow{4}{*}{$\bm{500}$}&hinge&$10.0\pm 4.8$ &$10.3\pm 4.6$&$11.6\pm 6.4$&$16.6\pm 5.4$\\
&robustSVM&$1.0\pm 0.3$ &$1.3\pm 0.4$&$2.0\pm 0.7$&$4.5\pm 1.3$\\
&conic1 &$7.2\pm 1.1$&$7.6\pm 1.4$&$11.7\pm 4.6$&$101.0\pm 28.7$\\
&hinge+conic1&$51.9\pm 10.5$ &$62.0\pm 19.8$&$86.0\pm 36.5$&$130.8\pm 32.0$\\
\hline
\multirow{4}{*}{$\bm{1,000}$}&hinge&$1.7\pm 0.5$ &$1.9\pm 0.5$&$2.5\pm 0.4$&$3.4\pm 1.3$\\
&robustSVM&$2.7\pm 0.9$ &$3.6\pm 1.1$&$5.5\pm 1.4$&$4.9\pm 2.2$\\
&conic1 &$10.4\pm 2.9$&$12.0\pm 4.8$&$18.5\pm 7.3$&$138.1\pm 69.3$\\
&hinge+conic1&$116.0\pm 36.2$ &$142.7\pm 61.0$&$320.5\pm 126.5$&$499.1\pm 110.4$\\
\hline
\end{tabular}
\end{center}
\end{table}

We observe that while \textbf{conic2} and the \textbf{0-1 loss} can be too expensive to use in a cross-validation setting, \textbf{conic1} is suitable for this use, as the computations can be carried out in under three minutes in all settings considered. Moreover, due to the reformulations discussed in \S\ref{sec:SVM}, \textbf{conic1} scales linearly with the number of datapoints $n$ and can be used if a large number of datapoints are collected. However, the performance degrades with the number of features $p$, and thus may be ineffective in large-dimensional datasets.

\subsection{Performance on real instances}\label{sec:statisticalReal}
We test the proposed methods in five datasets from the UCI machine learning repository \citep{Dua:2019}, whose names and dimensions (along with computational results) are reported in Table~\ref{tab:resultsReal}. We split each dataset randomly with 35\% of the observations in a training set, 35\% in a validation set and 30\% in the testing set. Additionally, we replicate the corruption used by \cite{jammal2020robust} and \cite{nguyen2013algorithms} in their experiments: given a noise parameter $\tau\in [0,0.5)$, we introduce outliers by flipping the label of each observation of the training and validation set with probability $\tau$ (the test set is not corrupted by outliers). We then perform cross-validation as described in \S\ref{sec:statisticalsynt}. For each dataset and value of the noise parameter $\tau$, we perform 20 random splits of train/validation/test and report average results and standard deviations over these replications.

The results are highly dependent on the specific dataset: for example \textbf{hinge} is better across all noise parameters for the sonar dataset, whereas \textbf{conic1} is better across all parameters for the breast cancer dataset. In general, no single estimator consistently outperforms or is outperformed by the others. For example, while \textbf{robustSVM} was consistently the worst estimator with synthetic data, it shows good performance in some cases and is tied for best (in terms of average performance) in settings with $\tau=0.2$. If anything, the computations show that real data are less structured than synthetic data, and one should be careful with extrapolations from the synthetic case to the real case. Nonetheless, in our computations, the combination of different estimators by method \textbf{ hinge +conic1} still results in the best performance overall. 

\begin{table*}[!h t b]
\caption{Results with real datasets, where $\tau$ is the probability of a label swap and each row shows averages and standard deviations over 20 splits of training/validation/testing. Column ``time" represents the total time required to solve 100 SVM problems (required for cross-validation). We depict in bold the best average out-of-sample accuracy for each dataset-noise combination.}
\label{tab:resultsReal}
\setlength{\tabcolsep}{2pt}
\resizebox{\textwidth}{!}{%
\begin{tabular}{c c c c | c | c c  c c}
\hline
\multirow{2}{*}{\textbf{name}}&\multirow{2}{*}{\textbf{n}}&\multirow{2}{*}{\textbf{p}}&\multirow{2}{*}{\textbf{method}}&\multirow{2}{*}{\textbf{time (s)}}&\multicolumn{4}{c}{\underline{\textbf{out-of-sample misclassification rate}}}\\
&&&&&$\bm{\tau=0.0}$&$\bm{\tau=0.1}$&$\bm{\tau=0.2}$&$\bm{\tau=0.3}$\\

\hline
\multirow{4}{*}{Breast cancer}&\multirow{4}{*}{196}&\multirow{4}{*}{36}&hinge&$0.6\pm 0.1$&$16.9\%\pm 4.4\%$&$18.3\%\pm 4.6\%$&$17.9\%\pm 5.0\%$&$19.4\%\pm 7.7\%$\\
&&&robustSVM&$0.5\pm 0.1$&$19.2\%\pm 6.1\%$&$18.4\%\pm 3.8\%$&$19.0\%\pm 6.7\%$&$25.9\%\pm 14.4\%$\\
&&&conic1&$19.9\pm 0.7$&$16.9\%\pm 4.1\%$&$\bm{17.3\%\pm 4.9\%}$&$\bm{17.5\%\pm 4.5\%}$&$\bm{18.5\%\pm 5.8\%}$\\
&&&hinge+conic1&$10.5\pm 0.4$&$\bm{16.7\%\pm 4.2\%}$&$18.2\%\pm 4.3\%$&$18.5\%\pm 5.8\%$&$20.4\%\pm 7.9\%$\\
\hline
\multirow{4}{*}{German credit}&\multirow{4}{*}{1,000}&\multirow{4}{*}{24}&hinge&$4.9\pm 0.5$&$25.0\%\pm 1.9\%$&$\bm{25.8\%\pm 2.3\%}$&$27.5\%\pm 2.0\%$&$\bm{29.4\%\pm 2.4\%}$\\
&&&robustSVM&$1.9\pm 0.2$&$25.9\%\pm 2.2\%$&$26.5\%\pm 2.2\%$&$27.7\%\pm 2.4\%$&$29.5\%\pm 2.2\%$\\
&&&conic1&$35.0\pm 2.8$&$28.7\%\pm 2.5\%$&$28.7\%\pm 2.4\%$&$28.8\%\pm 2.6\%$&$29.5\%\pm 2.5\%$\\
&&&hinge+conic1&$20.3\pm 1.5$&$\bm{24.9\%\pm 1.8\%}$&$\bm{25.8\%\pm 2.5\%}$&$\bm{27.1\%\pm 2.2\%}$&$30.0\%\pm 2.2\%$\\
\hline
\multirow{4}{*}{Immunotherapy}&\multirow{4}{*}{90}&\multirow{4}{*}{6}&hinge&$0.4\pm 0.0$&$\bm{20.5\%\pm 8.0\%}$&$24.1\%\pm 9.1\%$&$28.6\%\pm 10.4\%$&$34.6\%\pm 11.9\%$\\
&&&robustSVM&$0.2\pm 0.0$&$21.1\%\pm 7.0\%$&\bm{$21.8\%\pm 6.4\%$}&$\bm{22.7\%\pm 6.6\%}$&\bm{$26.1\%\pm 11.3\%$}\\
&&&conic1&$2.3\pm 0.1$&$22.1\%\pm 6.7\%$&$23.0\%\pm 6.4\%$&$24.3\%\pm 7.1\%$&$33.0\%\pm 20.5\%$\\
&&&hinge+conic1&$1.4\pm 0.0$&$\bm{20.5\%\pm 7.7\%}$&$24.5\%\pm 6.3\%$&$25.4\%\pm 6.5\%$&$34.1\%\pm 19.0\%$\\
\hline
\multirow{4}{*}{Ionosphere}&\multirow{4}{*}{351}&\multirow{4}{*}{34}&hinge&$1.0\pm 0.2$&$16.0\%\pm 3.3\%$&$19.2\%\pm 4.4\%$&$20.9\%\pm 5.0\%$&$24.0\%\pm 6.9\%$\\
&&&robustSVM&$0.9\pm 0.2$&$\bm{14.4\%\pm 3.2\%}$&$18.0\%\pm 3.3\%$&$21.0\%\pm 3.9\%$&$25.1\%\pm 5.7\%$\\
&&&conic1&$34.6\pm 3.1$&$15.3\%\pm 2.9\%$&$\bm{17.7\%\pm 3.9\%}$&$20.1\%\pm 5.5\%$&\bm{$23.7\%\pm 4.6\%$}\\
&&&hinge+conic1&$17.7\pm 1.1$&$15.6\%\pm 3.4\%$&$17.8\%\pm 3.6\%$&$\bm{19.4\%\pm 4.7\%}$&\bm{$23.7\%\pm 4.4\%$}\\
\hline
\multirow{4}{*}{Sonar}&\multirow{4}{*}{208}&\multirow{4}{*}{60}&hinge&$0.8\pm 0.2$&$25.4\%\pm 7.3\%$&$30.0\%\pm 7.6\%$&$\bm{32.7\%\pm 7.5\%}$&$\bm{42.3\%\pm 10.7\%}$\\
&&&robustSVM&$0.7\pm 0.1$&$28.9\%\pm 6.7\%$&$31.9\%\pm 7.0\%$&$33.2\%\pm 7.3\%$&$44.3\%\pm 10.3\%$\\
&&&conic1&$133.5\pm 12.1$&$30.3\%\pm 6.3\%$&$32.7\%\pm 5.2\%$&$37.4\%\pm 9.4\%$&$45.7\%\pm 8.0\%$\\
&&&hinge+conic1&$68.6\pm 6.3$&$\bm{25.2\%\pm 6.7\%}$&\bm{$\bm{29.2\%\pm 7.5\%}$}&$33.3\%\pm 7.5\%$&$42.5\%\pm 9.5\%$\\
\hline
\multirow{4}{*}{Average}&&&hinge&-&$20.8\%\pm 5.0\%$&$23.5\%\pm 5.6\%$&$25.5\%\pm 6.0\%$&\bm{$29.9\%\pm 7.9\%$}\\
&&&robustSVM&-&$21.9\%\pm 7.4\%$&$23.3\%\pm 7.1\%$&$\bm{24.7\%\pm 7.6\%}$&$30.2\%\pm 12.0\%$\\
&&&conic1&-&$22.7\%\pm 7.7\%$&$23.9\%\pm 7.7\%$&$25.6\%\pm 9.4\%$&$30.1\%\pm 13.8\%$\\
&&&hinge+conic1&-&$\bm{20.6\%\pm 6.5\%}$&$\bm{23.1\%\pm 6.7\%}$&$\bm{24.7\%\pm 7.8\%}$&${30.1\%\pm 12.8\%}$\\
\hline
\end{tabular}
}
\end{table*}

\section{Conclusions} \label{sec:conclusions}We study mixed-integer quadratic optimization problems, where binary variables are used to indicate the sign of continuous variables. Such problems arise, for example, in the context of classification in machine learning, where the sign of a given prediction indicates whether it is assigned a given label by a classifier. We provide an ideal description of the convex hull when the quadratic function is rank-one, strong inequalities to use for the general case, and discuss the implementation of the proposed formulations in the context of support vector machines. Our experiments demonstrate that the proposed formulations are able to compute better bounds than state-of-the-art mixed-integer optimization software with a fraction of the computational resources, and can be used directly as statistical estimators in problems with outliers or anomalies.

\section*{Acknowledgmenents}

Andr\'es G\'omez and Soobin Choi are supported, in part, by grants No. FA9550-24-1-0086 and FA9550-22-1-0369 from the Air Force Office of Scientific Research.  Shaoning Han is supported by the Ministry of Education, Singapore, under the Academic Research Fund Tier 1 (FY2024).

\bibliographystyle{abbrvnat}
\bibliography{main}

\appendix

\section{Proof of Proposition~\ref{corr:oneSided}.(a)}\label{sec:proofOneSided}

For convenience, we repeat the proposition to be proven concerning set 
\begin{equation*}
    \mathcal{X}_{\bm{d}}^-=\{(\bm{x},\bm{z},t)\in\mathbb{R}^n\times\{0,1\}^n\times\mathbb{R}:t\geq (\bm{d}^\top \bm{x})^2,~ x_i(1-z_i)\leq 0,~i=1,\ldots,n\}.
\end{equation*}
\begin{proposition} \label{corr:oneSidedRepeated}
  If $\bm{d} \geq \bm{0}$ (equivalently, if $\bm{d} \leq \bm{0}$), then
$$\cl\; \conv(\mathcal{X}_{\bm{d}}^-)=\Cbracket{(\bm{x},\bm{z},t)\in\mathbb{R}^n\times[0,1]^n\times\mathbb{R}:t\geq\frac{(\bm{d}^\top \bm{x})_+^2}{\min\
\Cbracket{1,\sum_{i\in \supp(\bm{d})}z_i }} + (\bm{d}^\top \bm{x})_-^2}.$$
\end{proposition}
    \begin{proof}
We assume that $\bm{d}>\bm{0}$ and discuss the extension to the case $\bm{d}\geq \bm{0}$ later. Define 
    \begin{equation*} 
    \mathcal{\widehat X}^-_{\bm{d}}=\Cbracket{(\bm{x},\bm{z},t)\in\mathbb{R}^n\times[0,1]^n\times\mathbb{R}:t\geq\frac{(\bm{d}^\top \bm{x})_+^2}{\min\{1,\sum_{i\in [n]} z_i\}} + (\bm{d}^\top\bm{x})^2_-}.
    \end{equation*}
    We show that for any $\bm{\alpha},\bm{\beta}\in\mathbb{R}^n,\gamma\in\mathbb{R}$, the two problems
    \begin{align}\label{prob:coro-mip}
        \min\; \bm{\alpha}^\top \bm{x} + \bm{\beta}^\top \bm{z} +\gamma t \text{ s.t. } (\bm{x},\bm{z},t)\in \mathcal{X}^-_{\bm{d}}
    \end{align}
    and  
    \begin{align}\label{prob:coro-relax}
        \min\; \bm{\alpha}^\top \bm{x} + \bm{\beta}^\top \bm{z} +\gamma t \text{ s.t. } (\bm{x},\bm{z},t)\in \mathcal{\widehat X}^-_{\bm{d}}
    \end{align}
    share the same optimal objective value and integer optimal solutions.
    If $\gamma < 0$, or if $\gamma=0$ and $\alpha_i\neq 0$ for some $i$, then both \eqref{prob:coro-mip} and \eqref{prob:coro-relax} are unbounded, applying similar arguments as those in Proposition \ref{propo:rank1}. Also, if $\gamma = 0$ and $\bm{\alpha}=\bm{0}$, then any solution with $z_i=\mathbbm{1}_{\{\beta_i<0\}}$ is optimal for both problems. Hence, it suffices to consider the case where $\gamma>0$ and $\bm{\alpha}\neq \bm{0}$. In such a case, if there is no $\eta\in\mathbb{R}$ such that $\bm{\alpha}=\eta \bm{d}$, then we can always find a vector $\bm{x}\in\mathbb{R}^n$ such that $\bm{\alpha}^\top \bm{x}<0$ and $\bm{d}^\top \bm{x}=0$, which implies that both \eqref{prob:coro-mip} and \eqref{prob:coro-relax} are unbounded. Therefore, the only remaining case to consider is $\bm{\alpha}=\eta\bm{d}$ for some $\eta\in\mathbb{R}$.
    
    With a change of variables, we may assume $\bm{d}=\bm{1}$ and $\gamma=1$. Then 
    \eqref{prob:coro-relax} can be rewritten as
\begin{equation}\label{prob:coro-relax2}
        \min~ \eta \bm{1}^\top \bm{x} + \bm{\beta}^\top \bm{z} + t \text{ s.t. } (\bm{x},\bm{z},t)\in \mathcal{\widehat X}^-_{\bm{d}}.
    \end{equation}
    Let $f$ be defined by   
    \begin{eqnarray*}
        f(\bm{z}) & := & \bm{\beta}^\top \bm{z} + \min_{y=(y)_+ + (y)_-} \eta ((y)_+ + (y)_-) +  \frac{(y)_+^2}{\min\{1,\sum_{i=1}^n z_i\}} + {(y)_-^2},
    \end{eqnarray*}
    with a change of variables $y=\bm{1}^\top \bm{x}$. Then, \eqref{prob:coro-relax2} reduces to
   $\min_{\bm{z}\in[0,1]^n} f(\bm{z}).$
    
    If $\eta > 0$, then the minimization above attains an optimal solution $y^*=(y^*)_-= -\eta/2$, that is, $f(z)=\bm{\beta}^\top \bm{z}-\eta^2/4$. Hence, \eqref{prob:coro-relax2} reduces to the problem $\min_{\bm{z}\in [0,1]^n} \bm{\beta}^\top \bm{z}-\eta^2/4$, which always achieves an integer optimal solution, say $\bm{z}^*$. Setting $x^*_1=y^*=-\eta/2$, $x^*_i=0$ for all $i\neq 1$, notice that the mixed integer solution $(\bm{x}^*,\bm{z}^*,t^*)$ with $t^*:=(\bm{1}^\top\bm{x}^*)^2$ is optimal for \eqref{prob:coro-relax2}. Also, one can check that the solution $(\bm{x}^*,\bm{z}^*,t^*)$ is feasible for \eqref{prob:coro-mip}, which implies that the optimal values of \eqref{prob:coro-mip} and \eqref{prob:coro-relax2} coincide.
    
    If $\eta\leq 0$, one can verify that for each $\bm{z}$, the point $y^*(\bm{z})=-\frac{\eta}{2}\min\{1,\sum_{i=1}^n z_i\}$ solves the minimization problem that defines $f(\bm{z})$. Thus, $f(\bm{z})=\bm{\beta}^\top \bm{z} - \frac{\eta^2}{4}\sum_{i=1}^nz_i$. Also, the problem $\min_{\bm{z}\in[0,1]^n} f(\bm{z})$ attains an integer optimal solution $\bar{\bm{z}}$ as established in Proposition \ref{propo:rank1}. If $\bar{\bm{z}}=\bm{0}$, then we can set $\bar{\bm{x}}=\bm{0}$ and $t=0$, which is optimal for \eqref{prob:coro-mip}. If $\bar{z}_j=1$ for some $j$, we can set $\bar{x}_j=y^*(\bar{z})$ and $\bar{x}_i=0$ for $i\neq j$. By letting $\bar{t}= (\bm{d}^\top\bar{\bm{x}})^2$, the solution $(\bar{\bm{x}},\bar{\bm{z}},\bar{t})$ is optimal for problem (\ref{prob:coro-relax2}) and is also feasible for problem (\ref{prob:coro-mip}). That is, we have shown that \eqref{prob:coro-relax} and \eqref{prob:coro-mip} share the same optimal value.

    Finally, extending these results to the general case $\bm{d}\geq \bm{0}$ follows from arguments similar to those in Theorem \ref{theo:hullGen}. We omit the details for brevity.
\end{proof}


\section{Proof of Proposition \ref{propo:cop-sdp}}\label{sec:proofEquiv}
For any cone $\mathcal{C}$, we denote the dual cone of $\mathcal{C}$ by $\mathcal{C}^*=\{\bm X^*:\innerProd{\bm X}{\bm X^*}\ge0\;\forall \bm X\in\mathcal{C} \}$. Define $\mathcal{CP}^n$ as the cone of $n$-dimensional \emph{completely positive matrices}, that is,  $\mathcal{CP}=\conv\left\{\bm x^\top \bm x: \bm x\in\R_+^n \right\}.$
Note that $\left(\mathcal{C}_+^n\right)^*=\mathcal{CP}^n.$ Proposition~\ref{propo:cop-sdp} is a direct corollary of Proposition~\ref{prop:sdp-exactness}. 
\begin{proposition}\label{prop:sdp-exactness}
    Assume $\bm X\in\mathcal{S}_+^n$. Then $\begin{pmatrix}
        t&\bm x^\top\\
        \bm x& \bm X
    \end{pmatrix}\in\mathcal{C}_+^{n+1}$ if and only if $\begin{pmatrix}
        t&\bm y^\top\\
        \bm y& \bm X
    \end{pmatrix}\succeq 0$ for certain $\bm y\le \bm x$. 
\end{proposition}
\begin{proof}
    Define $\mathcal{C}_L=\mathcal{C}_+^{n+1}\cap\mathcal{R}_L$ and $\mathcal{C}_R=\mathcal{S}^{n+1}_++\mathcal{R}_R$, where the sum of two sets is understood as Minkowski sum and 
    \[ \mathcal{R}_L=\left\{ \begin{pmatrix}
        t&\bm x^\top\\
        \bm x& \bm X
    \end{pmatrix}:\bm X\in\mathcal{S}_+^n  \right\} \text{ and }\mathcal{R}_R=\left\{ \begin{pmatrix}
        0&\bm x^\top\\
        \bm x& \bm 0
    \end{pmatrix}\in\R^{(n+1)\times(n+1)}:\bm x\in\R^n_+ \right\}. \]
    Then Proposition~\ref{prop:sdp-exactness} amounts to $\mathcal{C}_L=\mathcal{C}_R$ which is further equivalent to $\mathcal{C}_L^*=\mathcal{C}_R^*$ due to the closedness and convexity of $\mathcal{C}_L$ and $\mathcal{C}_R$. Since $\mathcal{C}_R\subseteq\mathcal{C}_L$, one has $\mathcal{C}_L^*\subseteq\mathcal{C}_R^*$. Thus, it suffices to prove $\mathcal{C}_R^*\subseteq\mathcal{C}_L^*$. 
    
    Note that  \begin{align*}
        \mathcal{C}_L^*=\left(\mathcal{C}_+^{n+1}\right)^*+\mathcal{R}_L^*=\mathcal{CP}^{n+1}+\mathcal{R}_L^*,
    \end{align*}
    and \begin{align}\label{eq:RHS-dual-cone}
        \mathcal{C}_R^*=\left(\mathcal{S}_+^{n+1}\right)^*\cap\mathcal{R}_R^*=\mathcal{S}_+^{n+1}\cap\mathcal{R}_R^*,
    \end{align}
    where 
    \[ \mathcal{R}_L^*=\left\{ \begin{pmatrix}
        0& \bm 0^\top \\
        \bm 0& \bm X
    \end{pmatrix}:\bm X\in\mathcal{S}_+^n  \right\} \text{ and }\mathcal{R}_R^*=  \left\{ \begin{pmatrix}
        t&\bm x^\top\\
        \bm x& \bm X
    \end{pmatrix}:\bm x\in\R_+^n  \right\}.\]   
    Taking an arbitrary $\hat{\bm X}= \begin{pmatrix}
        t&\bm x^\top\\
        \bm x& \bm X
    \end{pmatrix}\in\mathcal{C}_R^*$, one has $t\ge0$, $\hat{\bm{X}}\succeq0$ and $\bm x\ge0$ by \eqref{eq:RHS-dual-cone}. There are two cases:
    \begin{itemize}
        \item Case 1: $t=0$. In this case, one must have $\bm x=0$, implying $\hat{\bm X}\in\mathcal{R}_L^*\subseteq\mathcal{R}_L^*+\mathcal{CP}^{n+1}=\mathcal{C}_L^*.$
        \item Case 2: $t>0$. In this case, define $\hat{\bm X}=\underbrace{\begin{pmatrix}
            t&\bm x^\top\\
            \bm x& \frac{1}{t}\bm x\bm x^\top
        \end{pmatrix}}_{\hat{\bm X}^1}+\underbrace{\begin{pmatrix}
            0& \bm{0}^\top\\ \bm 0 &\bm X-\frac{1}{t}\bm x\bm x^\top
        \end{pmatrix}}_{\hat{\bm X}^2}$. One can see that $\hat{\bm{X}}^2\in\mathcal{R}_L^*$ and $\hat{\bm X}^1=\begin{pmatrix}
            \sqrt{t}\\\frac{\bm x}{\sqrt{t}}
        \end{pmatrix}\begin{pmatrix}
            \sqrt{t}\\\frac{\bm x}{\sqrt{t}}
        \end{pmatrix}^\top\in\mathcal{CP}^{n+1}$,  implying $\hat{\bm X}\in\mathcal{R}_L^*+\mathcal{CP}^{n+1}=\mathcal{C}_L^*.$
    \end{itemize}
    This shows $\mathcal{C}_R^*\subseteq\mathcal{C}_L^*$  and concludes the proof.
\end{proof}

\end{document}